\documentclass[a4paper, 11pt]{amsart}
\usepackage{epsfig,graphics,graphicx}
\usepackage{amssymb}
\usepackage{amsmath}
\usepackage[utf8]{inputenc}
\usepackage[english]{babel}
\usepackage{amsthm}
\usepackage{esint}

\newtheorem{theo}{Theorem} 
\newtheorem{exam}[theo]{Example}
\newtheorem{lem}[theo]{Lemma}
\newtheorem{prop}[theo]{Proposition}
\newtheorem{coro}[theo]{Corollary}
\newtheorem{defi}[theo]{Definition}
\newtheorem{quest}[theo]{Question}

\newtheorem{rem}[theo]{Remark}
\newtheorem*{theorA}{Theorem A}
\newtheorem*{theorB}{Theorem B}
\newtheorem*{theorC}{Theorem C}

\def\b{\mathbf b}
\def\g{\mathbf g}
\def\u{\mathbf u}
\def\v{\mathbf v}
\def\w{\mathbf w}
\def\R{\mathbb R}
\def\B{\mathbb B}
\def\D{\mathbb D}
\def\C{\mathbb C}
\def\Z{\mathbb Z}
\def\div{\operatorname{div}}
\def\supp{\operatorname{supp}}

\def\arg{\operatorname{arg}}
\def\Id{\operatorname{\textbf{Id}}}
\def\J{\operatorname{\textbf{J}}}
\def\curl{\operatorname{curl}}
\def\Exp{\operatorname{Exp}}
\def\tr{\operatorname{Tr}}

\title{Pointwise descriptions \\of nearly incompressible vector fields \\
with bounded curl}
\author{Albert Clop, Banhirup Sengupta}

\begin{document}
\maketitle

\begin{abstract}Among those nearly incompressible vector fields $\v:\R^n\to\R^n$ with $|x|\log|x|$ growth at infinity, we give a pointwise characterization of the ones for which $\curl\v= D\v-D^t\v$ belongs to $L^\infty$. When $n=2$ we can go further and describe, still in pointwise terms, the vector fields $\v:\R^2\to\R^2$ for which $|\div\v|+|\curl\v|\in L^\infty$.  \end{abstract}

\section{Introduction}

\noindent
Following \cite{Rei}, we will say that a  continuous vector field $\v:\R^n\to\R^n$ is of \emph{Reimann's type}, and write $\v\in Q$, if there is a constant $C_0\geq 0$ such that for each $x,h, k\in\R^n$ with $|h|=|k|\neq 0$ one has
$$\left|\frac{\langle \v(x+h)-\v(x),h\rangle}{|h|^2}-\frac{\langle \v(x+k)-\v(x),k\rangle}{|k|^2}\right|\leq C_0.$$
The best possible value of $C_0$ is denoted $\|\v\|_Q$. This class of vector fields was introduced by H.M. Reimann in \cite{Rei}.  Even though every Lipschitz vector field belongs to the $Q$ class, there exist many vector fields of Reimann type which are not Lipschitz. Indeed, every element of  $Q$ belongs to the Zygmund class. Thus, by the classical ODE theory, the autonomous initial value problem
$$\begin{cases}\frac{d}{dt}X(t,x)=\v(X(t,x)),\\X(0,x)=x.\end{cases}$$
has a well defined, unique flow of time-dependent solutions $X(t,x)$. Moreover, in the space variable $x$, this solution is a H\"older continuous homeomorphism. If $\v=\v(t,x)$ is not autonomous and also depends on time, then the same conclusion holds if one assumes $\sup_t\|\v(t,\cdot)\|_Q<\infty$. \\
\\
The relevance of Reimann's vector fields in Geometric Function Theory was first proven in \cite{Rei} with the quasisymmetry of the flow maps $x\mapsto X(t,x)$. At the same time, it is quite remarkable the fact that these maps enjoy a significant degree of Sobolev regularity in the space variable, as a consequence of the quasisymmetry. This fact puts Reimann's $Q$ class into a very narrow and unstable borderline: the one between the classical ODE theory and a much more recent result by Jabin \cite{J} (see also \cite{ACM}). Roughly, in the first theory Lipschitz vector fields are proven to produce bilipschitzian flows. The second theory refers to vector fields in the Sobolev space $W^{1,p}$ ($p<\infty$), and asserts that no Sobolev smoothness (even fractional) can be expected for their flow.\\
\\
Among the tools for proving the Sobolev regularity of the flow of a given $\v\in Q$, there is the following differential characterization from \cite[Theorem 3]{Rei},
\begin{equation}\label{Reimannequi}
\v\in Q\hspace{1cm}\Longleftrightarrow\hspace{1cm}S\v\in L^\infty( \R^n)\text{     and    }\frac{|\v(x)|}{|x|\,\log(e+|x|)}\leq C,
\end{equation}
as well as its quantitative formulation $\|\v\|_Q\simeq \|S\v\|_{L^\infty}$. Here $S\v$ denotes the traceless symmetric differential of $\v$,
$$S\v = \frac{D\v + D^t\v}2-\frac{\div\v}{n}\,\Id.$$
When $n=2$, $S\v$ reduces to $\overline\partial\v$, the classical Cauchy-Riemann derivative from complex analysis,
$$\overline\partial\v = \frac{(\partial_x+i\,\partial_y)(v^1+i\,v^2)}2\equiv\frac12\left(\begin{array}{c}\partial_xv^1-\partial_yv^2\\\partial_xv^2+\partial_yv^1\end{array}\right).$$
From \eqref{Reimannequi}, one deduces that if $\v\in Q$ then the flow map $x\mapsto X(t,x)$ is quasiconformal at every time. The authors address the interested reader to the monographs \cite{AIM} or \cite{IM2} for a self-contained background in quasiconformality. Roughly, quasiconformal maps are a relatively compact class of Sobolev homeomorphisms, and their trascendence goes beyond Geometric Function Theory to many areas in mathematics. In particular, when $n=2$ their optimal degree of Sobolev regularity can be obtained from Astala's Area Distortion Theorem \cite{A}.\\
\\
It turns out a similar situation occurs in several active scalar models, an apparently disconnected area. For instance, the planar Euler system for incompresible, inviscid fluids, in vorticity form
\begin{equation}\label{euler}
\begin{cases}
\omega_t + (\v\cdot\nabla)\omega=o\\
\v(t,\cdot)=\frac{i}{2\pi z}\ast \omega(t,\cdot)\\
\omega(0,\cdot)=\omega_0
\end{cases}
\end{equation}
was proven to be well posed by Yudovich \cite{Y} in the class of vector fields with bounded curl. More precisely, given a compactly supported $\omega_0:\R^2\to\R$ with $\omega_0\in L^\infty$, Yudovich \cite{Y} proved existence and uniqueness of a solution $\omega=\omega(t,z)$ of \eqref{euler} belonging to $L^\infty((0,\infty)\times \R^2)$. This, together with the incompressibility, provides us with a vector field $\v=\v(t,z)$ such that $\partial\v\in L^\infty((0,\infty)\times \R^2)$. Here $\partial\v$ denotes the complex derivative of the velocity field $\v$,
$$\partial\v= \frac{(\partial_x-i\,\partial_y)(v^1+i\,v^2)}2\equiv\frac12\left(\begin{array}{c}\partial_xv^1+\partial_yv^2\\\partial_xv^2-\partial_yv^1\end{array}\right)=\frac{\div\v+i\,\curl\v}2.$$
A similar situation is given in the aggregation model (in which the convolution kernel from \eqref{euler} is replaced by $\frac{1}{2\pi z}$). In analogy with Reimann, it was recently shown in \cite{CJ} that, at least for small times, vector fields $\v$ satisfying $\partial\v\in L^\infty$ admit a well defined flow which is Sobolev regular in the space variable, with a Sobolev exponent that may vary with time. In \cite{BN}, this result was improved and obtained a degree of Sobolev regularity for the flow for every time.    \\
\\
Although conditions $\overline\partial\v\in L^\infty$  and $\partial\v\in L^\infty$ may look analytically similar, they have a significant difference. In the first case, for a general non-autonomous $\v$, the flow map $X(t,\cdot)$ belongs to the Sobolev space $W^{1,p}_{loc}$ whenever
$$p<\frac{2}{1-\exp\left(-2\int_0^t\|\overline\partial\v(s,\cdot)\|_{L^\infty}\,ds\right)},$$
as a consequence of both Reimann's \cite{Rei} and Astala's \cite{A} Theorems. In contrast, this remains being an open problem in the second case. In accordance, it was conjectured in \cite{CJ} that if $\partial\v\in L^\infty$ then for each $t>0$ the flow map $X(t,\cdot)$ belongs to the Sobolev space $W^{1,p}_{loc}$ whenever
$$p<\frac{2}{1-\exp\left(-2\int_0^t\|\partial\v(s,\cdot)\|_{L^\infty}\,ds\right)}.$$
The asymptotic behavior of this conjecture as $t\to 0$ was proven to be the right one in \cite{CJ}. Moreover, when $\v$ arises from \eqref{euler}, this conjecture says that $p<\frac{2}{1-e^{-t\,\|\omega_0\|_{L^\infty}}}$. By the Sobolev embedding, this gives a H\"older exponent strictly below $e^{-t\,\|\omega_0\|_{L^\infty}}$, as shown by Bahouri and Chemin \cite{Chex}.  \\
\\
Geometric Function Theory has proven to be very useful in obtaining the optimal Sobolev regularity in Reimann's case, and therefore it is natural to try to face Euler's case with a similar scheme, as it was done in the works \cite{CJMO3, CJ}. In this paper, we continue this line of research by focusing our attention in the pointwise characterization of \cite[Theorem 3]{Rei}. We investigate the existence of similar pointwise characterizations of the condition $\partial\v\in L^\infty$, both in the plane and in higher dimensions. \\
\\
In the plane, we introduce the class $\bar{Q}$ of functions $\v: \R^2\to\R^2$ for which there is a constant $C_0\geq 0$ such that for each $x\in\R^2$ and every $h,k\neq0$ with $|h|=|k|$ one has
$$\left|\frac{\langle \v(x+h)-\v(x),\bar{h}\rangle}{|h|^2}-\frac{\langle \v(x+k)-\v(x),\bar{k}\rangle}{|k|^2}\right|\leq C_0.$$
Here $\bar{h}$ and $\bar{k}$ mean complex conjugates. By $\|\v\|_{\bar{Q}}$ we denote the best possible value of $C_0$. Similarly, we denote by $R$ the set of vector fields  
$\v:\R^2\to\R^2$ for which there is a constant $C_0\geq 0$ such that for each $x\in\R^2$, every $h,k\neq0$ with $|h|=|k|$, and every $\theta\in[0,2\pi]$, one has
$$\left|\frac{\langle \v(x+h)-\v(x),e^{i\theta}k\rangle}{|h||k|}-\frac{\langle \v(x+k)-\v(x),e^{i\theta}h\rangle}{|h||k|}\right|\leq C_0.$$
Again, $\|\v\|_{R}$ denotes the best possible constant $C_0$. Our first result is the following one. 

\begin{theorA}
Let $\v:\R^2\to\R^2$ be a continuous vector field. The following are equivalent:
\begin{itemize}
\item[(a)] $\v \in \bar{Q}$.
\item[(b)] $\v\in R$.
\item[(c)] $\v$ is differentiable a.e., $\partial \v\in L^\infty$, and $\frac{|\v(x)|}{|x|\,\log(e+|x|)}\leq C$.
\end{itemize}
If one of them holds true, then $\|\v\|_{\bar{Q}}\simeq \|\v\|_R\simeq\|\partial\v\|_{L^\infty}$.
\end{theorA}

\noindent
The presence of complex conjugation in the definition of $\bar{Q}$ prevents us from extending it to higher dimensions, at least trivially. Extending the definition of $R$ to $\R^n$, $n\geq 2$, seems not an easy task either, because the set of rotations to be included is not obvious (see Lemma  \ref{Rnfailure}).  It turns out that one may still get some $L^\infty$ estimates by removing all rotations, even in higher dimensions. Namely, let us introduce $R_0$ as the class of vector fields $\v:\R^n\to\R^n$ for which there is $C_0$ such that for each $x\in \R^n$ and each $h,k$ with $|h|=|k|\neq 0$ one has
$$\left|\frac{\langle \v(x+h)-\v(x),k\rangle}{|h||k|}-\frac{\langle \v(x+k)-\v(x),h\rangle}{|h||k|}\right|\leq C_0.$$
As usually, $\|\v\|_{R_0}$ denotes the best possible constant $C_0$.  

\begin{theorB}
Let $\v\in R_0$. Then the distribution $D\v-D^t\v$ belongs to $L^\infty$, and 
$$\|D\v-D^t\v\|_{L^\infty}\leq C\,\|\v\|_{R_0}.$$
for some constant $C>0$.\end{theorB}

\noindent
As it was the case for $Q$, $\bar{Q}$ or $R$, the elements of $R_0$ belong as well to the Zygmund class. However, when $n=2$ the class $R_0$ is much larger than $R$, and one cannot guarantee its elements to be differentiable a.e.. This makes it more difficult to find higher dimensional counterparts to Theorem A.  In the present paper we solve this by asking $\v$ to be nearly incompressible, that is, $\div\v\in L^\infty$. This allows to state the above mentioned counterpart, which is based in the differential operator
$$A\v=\frac{D\v-D^t\v}2+\frac{\div\v}n\,\Id.$$
Note that for $n=2$ one has $A\v\equiv\partial\v$.

\begin{theorC}
Let $\v:\R^n\to\R^n$ be a continuous vector field.
\begin{itemize}
\item[(a)] If $\v\in \R_0$ and $\v$ is nearly incompressible, then $\v$ is differentiable a.e. and the estimate
$$\|A\v\|_{L^\infty}\leq C\,(\|\div\v\|_{L^\infty}+\|\v\|_{R_0})$$
holds.
\item[(b)] If $A\v\in L^\infty$ and $\frac{|v(x)|}{|x|\,\log(e+|x|)}\leq C$ then $\v\in R_0$ and 
$$\|\v\|_{R_0}\leq C\,\|A\v\|_{L^\infty}.$$
\end{itemize}
\end{theorC}

\noindent
As in Reimann's setting, one of the main tools here is the fact that if $\v$ is a compactly supported vector fields with $A\v\in L^\infty$ then $\v$ has $BMO$ derivatives and, in particular, it is differentiable a.e. (see Lemma \ref{curldiv}). For this reason, here one can relax the assumption $\div\v\in L^\infty$ to $\div\v\in L^p$ for some $p>n$. On the other hand, as a possible application, the above result can be used to describe in a pointwise way, among all the solutions to the Euler system of equations, the ones with bounded curl.  \\
\\
The paper is structured as follows. In Section \ref{poisson} we recall some basic facts about Poisson integrals that will be used in the rest of the paper. In Section \ref{barQsection} we prove $(a)\Leftrightarrow(c)$ from Theorem A. In Section \ref{Rsection} we prove $(b)\Leftrightarrow(c)$ from Theorem A. In Section \ref{R^nsection} we prove Theorems B and C. 

\textbf{Notation}. Bold letters like $\b,\u,\v,\w,\g$ denote vector valued functions. After identifying planar vectors with complex numbers, the inner product in $\R^2$ can be represented as $\langle z,w\rangle=\Re(z\bar{w})$, where $\Re$ denotes real part and $\bar{w}$ stands for the complex conjugate of $w$, that is, if $w=(w_1, w_2)$ then $\bar{w}=(w_1, -w_2)$. If $A\simeq B$ then there is a constant $C\geq 0$ such that $\frac{B}{C}\leq A\leq CB$.

\textbf{Acknowledgements}. The authors warmly thank Artur Nicolau for letting us know about Lemma \ref{higherorder}. Both authors are partially supported by projects MTM2016-81703-ERC, MTM2016-75390 (spanish Government) and 2017SGR395 (catalan Government).

\section{Preliminaries}\label{poisson}

\noindent
In this section we recall some fundamental facts concerning harmonic functions on the upper half space. We refer the interested reader to \cite{St} for a more detailed review on this. We will be working with functions defined on $\R^{n+1}_+$, where points are represented as $(x,y)$ with $x\in\R^n$ and $y>0$. Let us recall that a function $u:\R^{n+1}_+\to\R$ is said to be harmonic if 
$$\Delta u (x,y)=0$$
where $\Delta = \Delta_x+\partial_{yy}^2=\sum_{i=1}^n\partial^2_{x_i,x_i}+ \partial_{yy}^2$. A typical way of constructing harmonic functions on the upper half space is through the Poisson integral of a function $g:\R^n\to\R$,
$$u(x,y)= P_y\ast g(x)=\int_{\R^n}P_y(x-z)\,g(z)\,dz$$
where
$$P(z,y)=P_y(z)=\frac{c_n y}{(|z|^2+y^2)^\frac{n+1}2}$$ 
is the Poisson kernel. Above, the constant $c_n$ is chosen so that $\|P_n\|_{L^1(\R^n)}=1$. For a vector valued $\g:\R^n\to\R^m$, then one interprets $\u=P_y\ast\g:\R^{n+1}_+\to\R^m$ componentwise. In either case, one often says that $u$ is the \emph{Poisson integral of $g$}, and that $g$ represents $u$'s \emph{boundary values}. The following result explains the latter terminology. 

\begin{prop}
If $g\in C_c(\R^n)$ then $u=P_y\ast g$ is the only solution to the Dirichlet problem
$$
\begin{cases}
\Delta u=0&\R^{n+1}_+\\u(\cdot, 0)=g&\R^n.
\end{cases}
$$
\end{prop}
\begin{proof}
From$$
\partial^2_{yy}P(z,y)=(n+1)P_y(z)\,\frac{ -3|z|^2+ny^2}{(|z|^2+y^2)^2}\hspace{1cm}\Delta_zP_y(z)=(n+1)\,P_y(z)\,\frac{3|z|^2-ny^2}{(|z|^2+y^2)^2}
$$
it is immediate that $\Delta_z P_y(z)+\partial_{yy}^2P_y(z)=0$ and so $P_y(z)$ is harmonic on $\R^{n+1}_+$. As a consequence, $u$ is harmonic on $\R^{n+1}_+$. About the boundary condition, it suffices to observe that $P_y$ is an approximation of unity in $\R^n$, so one has $P_y\ast g\to g$ uniformly as $y\to 0$. In particular, $u$ is continuous on $\overline{\R^{n+1}_+}$ and $u(x,0)=g(x)$ for every $x\in\R^n$. Uniqueness follows from the maximum principle for harmonic functions.
\end{proof}

\noindent
One may ask if there are other harmonic functions in $\R^{n+1}_+$ that are not representable as $P_y\ast g$ for some $g$. The theory of Hardy spaces helps in this direction. Note that one may also define $P_y\ast g$ even when $g$ is a measure.

\begin{prop}\label{hardy}
Let $u:\R^{n+1}_+\to\R$ be harmonic. 
\begin{itemize}
\item Given $1<p\leq\infty$, there is $g\in L^p(\R^n)$ such that $u=P_y\ast g$ if and only if $\sup_y \|u(\cdot, y)\|_{L^p}<\infty$, and moreover in this case one has $\|g\|_{L^p}=\sup_y \|u(\cdot, y)\|_{L^p}$.
\item There is a finite Borel measure $\mu$ on $\R^n$ with $u=P_y\ast\mu$ if and only if $\sup_y \|u(\cdot, y)\|_1<\infty$, and moreover in this case one has $\|\mu\|=\sup_y \|u(\cdot, y)\|_1$. Furthermore, if $u>0$ then $\mu$ is non-negative.
\end{itemize}
\end{prop}

\noindent
Poisson integrals of $BMO$ functions can also be characterized, but its description involves a completely different quantity, as stated in the following Theorem by Carleson. Let us remind that $g:\R^n\to\R$ belongs to the $BMO$ class if 
$$\|g\|_\ast=\sup\left\{\frac{1}{|B|}\int_B\left|g-\frac{1}{|B|}\int_Bg\right|; B\subset\R^n\text{ is a ball}\right\}<\infty.$$

\begin{theo}\label{carlesonthm}
Let $u:\R^{n+1}_+\to\R$ be harmonic. Then $u=P_y\ast g$ for some $g\in BMO(\R^n)$ if and only if 
$$
\|u\|_{\ast\ast}=\sup_{x_0\in\R^n,\delta>0}\frac{1}{|B(x_0,\delta)|}\int_0^\delta\int_{B(x_0,\delta)}(|D_xu(x,y)|^2+|\partial_yu(x,y)|^2)\,dx\,y\,dy <\infty.
$$
Moreover, in case this happens, then $\|u\|_{\ast\ast}\simeq \|g\|_\ast$ with universal constants.
\end{theo}

\noindent
For a non continuous function $g$, calling it to be the \emph{boundary values} of $P_y\ast g$ requires some explanation. Let us remind that the limit $\lim_{y\to 0}u(x,y)= g(x)$ is said to be taken \emph{nontangentially at the point $x$} if and only if it happens when $(x,y)$ move within a cone with vertex $x$.

\begin{prop}\label{nontangential}
Let $g\in L^p(\R^n)$. 
\begin{itemize}
\item If $1\leq p\leq\infty$, then $P_y\ast g\to g$ nontangentially at almost every point. 
\item If $1<p<\infty$, then $\|P_y\ast g-g\|_{L^p}\to 0$ as $y\to 0$.  
\item If $p=1$ or $p=\infty$ then there exists $g\in L^p(\R^n)$ such that $\|P_y\ast g-g\|_{L^p}\nrightarrow 0$ as $y\to 0$.
\end{itemize}
\end{prop}

\noindent
In the case of Borel measures, the situation is significantly different. To see this, if $
 \delta_0$ is the Dirac Delta then $u(\cdot,y)=P_y\ast\delta_0=P_y$ so that $u(x,0)\equiv 0$. It turns out the following is true. 

\begin{prop}
If $g$ is a finite Borel measure, singular w.r.t. $dx$, then $P_y\ast g$ has nontangential limit $0$ almost everywhere.
\end{prop}

\noindent
Combining propositions \ref{hardy} and \ref{nontangential}, one sees that every bounded harmonic function $u:\R^{n+1}_+\to\R$ is precisely of the form $u=P_y\ast g$ for some $g\in L^\infty(\R^n)$, and moreover $u(\cdot,y)$ converges nontangentially to $g$ at almost every point. It is interesting to note that there is some control as well on the first order derivatives of $u$.

\begin{lem}\label{easyLinfty}
If $g\in L^\infty(\R^n)$ then
$$
\aligned
\|P_y\ast g\|_{L^\infty}&\leq \|g\|_{L^\infty} \\
\|\partial_y(P_y\ast g)\|_{L^\infty}&\leq n\,\frac{\|g\|_{L^\infty} }{y}\\
\|D_x(P_y\ast g)\|_{L^\infty}&\leq \frac{n+1}2\,\frac{\|g\|_{L^\infty} }{y}
\endaligned$$  
\end{lem}
\begin{proof}
First, one easily sees that $|P_y\ast g(x)|\leq \|P_y\|_1\cdot\|g\|_{L^\infty}=\|g\|_{L^\infty}$ since $\|P_y\|_{L^1(\R^n)}=1$. Secondly, direct calculation shows that
$$
\partial_yP(z,y)= \frac{P_y(z)}{y}\,\frac{|z|^2-ny^2}{|z|^2+y^2}\hspace{2cm}D_zP_y(z)=\frac{P_y(z)}y\,\frac{-(n+1)\,yz}{|z|^2+y^2} 
$$
Thus, $|\partial_yP_y(z)|\leq\frac{n\,P_y(z)}{y} $ and hence $\left|(\partial_yP_y)\ast g(x)\right|\leq n\, \frac{\|g\|_{L^\infty}}{y} $. The bound for the spatial derivative follows in the same way, after observing that $|D_zP_y(z)|\leq \frac{n+1}2\,\frac{P_y(z)}{y}$.
\end{proof}

\noindent
Lemma \ref{easyLinfty} motivates the introduction of the class $B$ of \emph{harmonic Bloch functions}, which consists of functions $u:\R^{n+1}_+\to\R$ that are harmonic in $\R^{n+1}_+$ and whose gradient blows up as $y\to0$ like $\frac1y$, that is,
$$u\in B\hspace{1cm}\Longleftrightarrow\hspace{1cm}u\text{ is harmonic and }\|u\|_B=\sup_{\R^{n+1}_+}\,y(|D_xu(x,y)|+|\partial_yu(x,y)|)<\infty.$$
Vector valued harmonic Bloch functions are defined componentwise. Examples of harmonic Bloch functions are, for instance, Poisson integrals of $L^\infty$ functions, as shown in Lemma \ref{easyLinfty}. It turns out Poisson integrals of $BMO$ functions also belong to the Bloch class. 

\begin{lem}\label{poissonBMO}
If $g\in BMO$ then $P_y\ast g\in B$, and moreover
 $$
 \aligned
  \|\partial_y (P_y\ast g)\|_{L^\infty}&\leq \frac{C(n)\,\|g\|_\ast}y\\
\|D_x(P_y\ast g)\|_{L^\infty}&\leq \frac{C(n)\,\|g\|_\ast}{y}\\
 \endaligned$$
 \end{lem}
\begin{proof}
One can find a proof in \cite[p. 86-87]{RR} or also in \cite[Lemma 1.1]{FJN}. We sketch the latter here for the reader's convenience. Denote $u(x,y)=P_y\ast g(x)$. Then $u$ is harmonic in the upper half space, and therefore all its partial derivatives are harmonic as well. By the mean value property, if $r=\frac{y_0}4$ then
$$\partial u(x_0,y_0)=\fint_{|x-x_0|^2+|y-y_0|^2<r^2} \partial u(x,y)\,dx\,dy$$
at any point $(x_0, y_0)\in\R^{n+1}_+$. Here $\partial$ denotes any element of the set $\{\partial_{x_1},\dots,\partial_{x_n},\partial_y\}$. We now observe that
$$\aligned
|\partial u(x_0,y_0)|^2
&=\left|\fint_{|x-x_0|^2+|y-y_0|^2<r^2} \partial u(x,y)\,dx\,dy\right|^2\\
&\leq \fint_{|x-x_0|^2+|y-y_0|^2<r^2} |\partial u(x,y)|^2\,dx\,dy\\
&\leq \fint_{|x-x_0|^2<r^2}\,\fint_{[3y_0/4, 5y_0/4]}|\partial u(x,y)|^2\,dy\,dx\\
&\leq \fint_{|x-x_0|^2<r^2}\,\frac{8}{3y_0^2}\int_{[3y_0/4, 5y_0/4]}|\partial u(x,y)|^2\,y\,dy\,dx\\
&\leq \frac{c}{r^2} \fint_{|x-x_0|^2<(5r)^2}\,\int_{[0, 5r]}|\partial u(x,y)|^2\,y\,dy\,dx\leq \frac{c}{r^2} \,\|u\|_{\ast\ast}^2\leq \frac{c}{r^2}\,\|f\|_\ast^2=\frac{c}{y_0^2}\,\|f\|_\ast^2\endaligned$$
as claimed. 
\end{proof}

\noindent
It is very rellevant for this paper the blow-up at the boundary of higher order derivatives of Poisson integrals. In this direction, we have the following fact from \cite[Appendix]{St}. 

\begin{lem}\label{higherorder}
If $u:\R^{n+1}_+\to\R$ is harmonic, then
$$
\sup_{(x,y)\in \R^{n+1}_+}\,\left(\sup_{1\leq i_1\leq\dots\leq i_k\leq n+1} y^k\,|\partial^k_{x_{i_1}\dots x_{i_k}} u(x,y)|\right)\leq C(n,k)\,\sup_{(x,y)\in \R^{n+1}_+}\,\sup_{1\leq i\leq n+1} y \,|\partial_{x_i} u(x,y)|.
$$
\end{lem}

\noindent
In other words, the blow-up of the first order derivatives roughly determines that of the higher order ones. In particular, if $u$ is a harmonic Bloch function and $Hu(x,y)$ denotes its $(n+1)$-dimensional Hessian,
$$Hu(x,y)=\left(\begin{array}{cc}D^2_xu(x,y)&D_x\partial_yu(x,y)\\D_x\partial_yu(x,y)&\partial^2_{yy}u(x,y)\end{array}\right)$$
then one has
\begin{equation}\label{blochy2}
y^2\,|Hu(x,y)|\leq  C(n)\, \|u\|_B.
\end{equation}
It turns out that the bound \eqref{blochy2} may be significantly improved if $u$ is the harmonic extension of a function in the Lipschitz class. Recall that $g:\R^n\to\R$ is Lipszhitz if 
$$\|g\|_{Lip}=\inf\left\{C\geq 0: |g(x)-g(y)|\leq C|x-y|\text{ for every }x,y\in\R^n\right\}<\infty.$$
Lipschitz functions are also characterized by having bounded derivatives. Thus, if $g\in Lip$ and $u=P_y\ast g$ then $D_xu=P_y\ast Dg$ and therefore combining Lemmas \ref{easyLinfty} and \ref{higherorder} one gets
\begin{equation}\label{blochy1}
y\,|Hu(x,y)|\leq  C(n)\, \|Dg\|_{L^\infty}.
\end{equation}
which certainly improves \eqref{blochy2}. Let us recall that $g:\R^n\to\R$ is an element of $Z$ if and only if
$$\|g\|_Z=\inf\left\{C\geq 0: |g(x+h)+g(x-h)-2g(x)|\leq C|h|\text{ for every }x,h\in\R^n\right\}<\infty.$$
For instance, if $g$ has distributional derivatives $Dg\in BMO$ then $g\in Z$. The Zygmund class is a little larger than the Lipschitz class $Lip$. Indeed, one may think that $Z$ is to $Lip$ what $BMO$ is to $L^\infty$. Thus, the following result has an easy proof for functions in $Lip$, and a more complicate one for functions in $Z$.

\begin{lem}\label{poissonZ}
Let $g\in L^\infty(\R^n)$, and $u=P_y\ast g$. Let $\nabla u=(D_xu,\partial_yu)$ denote the $(n+1)$-dimensional gradient of $u$. Then $g\in Z$ if and only if $\nabla u \in B$, and moreover
$$\frac1C\,\|g\|_Z\leq \|\nabla u\|_B\leq C\,\|g\|_Z$$
for some constant $C>0$.
\end{lem}

\noindent
A proof of this fact can be found in  \cite[p. 146]{St}. As a consequence, if $g\in Z$ and $u=P_y\ast g$ then 
Lemma \ref{poissonZ} tells that 
\begin{equation}\label{hessian1/y}
y\,|Hu(x,y)|\leq  C\,\|g\|_Z ,
\end{equation}
which is better than \eqref{blochy2}. Moreover, one can combine this with Lemma \ref{higherorder} and obtain that $y^k |\nabla^{k+1}u(x,y)|$ is bounded by a multiple of $\|g\|_Z$, for every $k=1,2,\dots$. Inequality \eqref{hessian1/y} can be proven, for instance, if $g$ is a function with $Dg\in BMO$. That proof requires the help of the classical $BMO-H^1$ duality (see also \cite[p. 263, top Corollary]{Rei}). However,  functions with $BMO$ derivatives belong to the Zygmund class. For this reason, we prefered to state Lemma \ref{poissonZ} and use the notion of harmonic Bloch gradients, which characterizes the class of Zygmund functions and at the same time allows for a more precise constant in the inequality.\\
\\
Finally, we include in this section the following result, which will be repeatedly used in the rest of the paper, and whose proof is implicit in \cite{IM}. Let us recall that if $\b:\R^n\to\R^n$ is a vector field, then one defines the divergence and the curl of $\b$, respectively, as
$$\div\b =\tr(D\b)\hspace{1cm}\curl\b = D\b-D^t\b$$

\begin{lem}\label{curldiv}
Let $1<p<\infty$. If $\b:\R^n\to\R^n$ is continuous and compactly supported, and $\curl\b,\div\b\in L^p(\R^n)$, then also $D\b\in L^p(\R^n)$, with
$$\|Db\|_{L^p}\leq C\,(\|\div\b\|_{L^p}+\|\curl\b\|_{L^p}).$$ 
If the assumptions hold with $p=\infty$, then one has $D\b\in BMO$, and
$$\|Db\|_{\ast}\leq C\,(\|\div\b\|_{L^p}+\|\curl\b\|_{L^p}).$$
\end{lem}
\begin{proof}
We write the proof for the reader's convenience. When $n=2$, the assumptions say that $\b$ has complex derivative $\partial\b=\frac{\div\b+i\curl\b}2$ in $L^p$. Since $\b$ is continuous and compactly supported, we can write $\b = \frac{1}{\pi\bar{z}}\ast (\partial\b)$, whence $\overline\partial\b = p.v.\frac{-1}{\pi\bar{z}^2}\ast(\partial\b)$. But the convolution with $p.v.\frac{-1}{\pi\bar{z}^2}$ defines a Calder\'on-Zygmund operator, and thus $\overline\partial\b\in L^p$ (or $BMO$, if $p=\infty$) with $\|\overline\partial\b\|_{L^p}\leq C\,\|\partial\b\|_{L^p}$ (resp. $\|\overline\partial\b\|_{\ast}\leq C\,\|\partial\b\|_{L^\infty}$) as claimed. \\
\\
When $n>2$ the proof is a little bit delicate. We start by reminding that the second derivatives of a function $v$ vanishing at infinity can be recovered from its laplacian $\Delta v$ through the second order Riesz transforms,
$$
\frac{\partial^2v}{\partial x_j\partial x_k}  = - R_jR_k\Delta v , \hspace{1cm} j,k=1,\dots, n.
$$
where $\widehat{R_j v}(\xi)=-i\frac{\xi_j}{|\xi|}\,\widehat{v}(\xi)$ at the Fourier side. As Calder\'on-Zygmund operators, one has again that $R_j: L^p\to L^p$ is bounded if $1<p<\infty$, and that $R_j: L^\infty\to BMO$ is bounded. We now proceed first with the proof for $p\in(1,\infty)$. Since $\b$ is continuous and compactly supported, the Poisson equation
$$\Delta \u =\b$$
has a unique solution $\u:\R^n\to\R^n$ vanishing at infinity. In particular, the distributional Hessian matrix $H\u$ of the solution $\u$ has all its entries in $L^s$ and $\|H\u\|_s\leq C\,\|\b\|_s$, for every $s\in (1,\infty)$. We now decompose $\b$ as follows,
\begin{equation}\label{decomp}
\b=\nabla \div\u + \div \curl\u
\end{equation}
where we recall that $\curl \u=D\u - D^t\u$ is a matrix valued field. This is, indeed, the Hodge decomposition of $\b$ as the sum of a curl free vector field (i.e. $\nabla\div\u$) and a divergence free field (i.e. $\div\curl\u$). We now observe that $\curl\u$ solves the following Poisson equation,
\begin{equation}\label{eqcurl}\Delta(\curl \u)=\curl\b\end{equation}
because $\Delta(\curl \u)=\curl (\Delta\u)$. In particular, if $\curl\b\in L^p$ then the same holds for the hessian $H(\curl\u)$, and moreover $\|H(\curl\u)\|_{L^p}\leq C\,\|\curl\b\|_{L^p}$. Similarly, $\div\u$ solves the Poisson equation
\begin{equation}\label{eqdiv}\Delta(\div\u)=\div\b\end{equation}
because $\Delta(\div\u)=\div(\Delta\u)$. This shows that if $\div\b$ belongs to $L^p$ then also the hessian $H(\div\u)$ does, and we have the bound $\|H(\div\u)\|_{L^p}\leq C\,\|\div\b\|_{L^p}$.  Summarizing, if both $\curl\b,\div\b\in L^p$, then both hessians $H(\curl\u)$ and $H(\div\u)$ have $L^p$ entries, whence both terms in the right hand side of \eqref{decomp} belong to the homogeneous Sobolev space $\dot{W}^{1,p}$,  and
$$\aligned\|\b\|_{\dot{W}^{1,p}}&\leq \|\nabla \div\u \|_{\dot{W}^{1,p}}+ \|\div \curl\u\|_{\dot{W}^{1,p}}\\
&\leq\|H(\div\u)\|_{L^p}+\|H(\curl\u)\|_{L^p}\\
&\leq C\|\div\b\|_{L^p}+C\|\curl\b\|_{L^p}\endaligned$$
so the claim follows if $1<p<\infty$. In case that $\curl\b,\div\b\in L^\infty$, then the proof follows similarly, with the only difference that now $\curl\u$ and $\div\u$ have distributional hessian in $BMO$ instead, and therefore both terms in \eqref{decomp} have first order derivatives in  $BMO$, so $\b$ also does. \\
\\
It just remains to prove \eqref{decomp}, which we do by direct calculation,
$$\aligned
\nabla\div\u &+ \div \curl\u =\\
&= \left(\begin{array}{c}\partial_{x_1}\div\u\\\vdots\\\partial_{x_n}\div\u\end{array}\right)+\div\left(\begin{array}{cccc}
0&\partial_{x_2}u^1-\partial_{x_1}u^2&\dots&\partial_{x_n}u^1-\partial_{x_1}u^n\\
\partial_{x_1}u^2-\partial_{x_2}u^1&0&\dots&\partial_{x_n}u^2-\partial_{x_2}u^n\\
\vdots&\vdots&\ddots&\hdots\\
\partial_{x_1}u^n-\partial_{x_n}u^1&\partial_{x_2}u^n-\partial_{x_n}u^2&\dots&0\\
\end{array}\right)\\
&=\left(\begin{array}{c}\sum_j\partial^2_{x_1x_j}u^j\\\sum_j\partial^2_{x_2x_j}u^j\\\vdots\\\sum_j\partial^2_{x_nx_j}u^j\end{array}\right)+
\left(
\begin{array}{c}
\sum_{j\neq 1}\partial^2_{x_jx_j}u^1-\partial_{x_1}\sum_{j\neq 1}\partial_{x_j}u^j\\
\sum_{j\neq 2}\partial^2_{x_jx_j}u^2-\partial_{x_2}\sum_{j\neq 2}\partial_{x_j}u^j\\
\vdots\\
\sum_{j\neq n}\partial^2_{x_jx_j}u^2-\partial_{x_2}\sum_{j\neq n}\partial_{x_j}u^j
\end{array}
\right)=\Delta\u.
\endaligned$$
This is legitimate for $\u$ because it has locally integrable second order derivatives.
\end{proof}

\section{The planar setting: the class $\bar{Q}$}\label{barQsection}

With the spirit of finding a counterpart to Reimann's $Q$ class, we introduce a class $\bar{Q}$ consisting of functions $\b:\R^2\to\R^2$ such that there is $C>0$ with
$$
\|\b\|_{\bar{Q}}=\sup_{z\in\R^2}\sup_{|h|=|k|\neq 0}\left|\frac{\langle \b(z+h)-\b(z),\bar{h}\rangle}{|h|^2}-\frac{\langle \b(z+k)-\b(z),\bar{k}\rangle}{|k|^2}\right|<\infty
$$
It is not hard to see that Lipszchitz functions are elements of $\bar{Q}$. Also, arguing as in \cite{Rei}, one can show that the elements of $\bar{Q}$ are, at every time $t$, elements of the Zygmund $Z$ class. 

\begin{prop}\label{LipsbarQsZ}
If $\b:\R^2\to\R^2$, then one has
$$\|\b \|_Z\leq C\,\|\b\|_{\bar{Q}}\leq C\, \|\b \|_{Lip}.$$ 
In particular, Lipschitz vector fields belong to $\bar{Q}$, and elements of $\bar{Q}$ are Zygmund vector fields. Also, if $\b\in\bar{Q}$ then it holds that
$$
\left|\frac{\langle \b(z+h)-\b(z),\bar{h}\rangle}{|h|^2}-\frac{\langle \b(z+k)-\b(z),\bar{k}\rangle}{|k|^2}\right|\leq C\,\left(1+\left|\log\frac{|h|}{|k|}\right|\right)
$$
for all pairs $h,k\neq 0$, and with $C\leq c\,\|\b\|_{\bar{Q}}$, where $c$ is a constant independent of $\b$.
\end{prop} 

\noindent
The proof of the above result follows the lines of \cite{Rei}, and therefore we omit it. The interested reader is adressed to Propositon \ref{zygmund} below, whose proof is very similar. In the following lemma we give a rather descriptive necessary condition for smooth elements of $\bar{Q}$.

\begin{lem}\label{smoothbarQ}
Let $\b:\R^2\to\R^2$ be smooth. If $\b\in \bar{Q}$, then $\partial \b\in L^\infty$ and
$$\|\partial \b\|_{L^\infty}\leq \frac12\,\|\b\|_{\bar{Q}}.$$
\end{lem}
\begin{proof}
In complex coordinates, the Taylor expansion of $\b$ at a differentiability point $z\in\R^2$ looks as follows,
$$\b(z+h)-\b(z)=\partial \b(z)h+\overline\partial \b(z)\,\bar{h}+o(|h|).$$
Hence, if we now take inner product with $\bar{h}$, we obtain
$$\aligned
\langle \b(z+h)-\b(z),\bar{h}\rangle 
&= \langle \partial \b(z)h+\overline\partial \b(z)\,\bar{h},\bar{h}\rangle+\langle o(|h|), \bar{h}\rangle\\
&= \Re ((\partial \b(z)h+\overline\partial \b(z)\,\bar{h}),h)+\langle o(|h|), \bar{h}\rangle\\
&= \Re ((\partial \b(z)h^2))+\Re((\overline\partial \b(z))\,|h|^2)+\langle o(|h|), \bar{h}\rangle\\
&= \Re ((\partial \b(z)h^2))+\Re((\overline\partial \b(z))\,|h|^2)+\langle o(|h|), \bar{h}\rangle\\\endaligned$$
whence
$$\aligned
\left|\frac{\langle \b(z+h)-\b(z),\bar{h}\rangle}{|h|^2}-\frac{\langle \b(z+k)-\b(z),\bar{k}\rangle}{|k|^2}\right|&=\Re\left(\partial \b(z)\left(\frac{h^2}{|h|^2}-\frac{k^2}{|k|^2}\right)\right)\\&+\frac{\langle o(|h|),h\rangle}{|h|^2}+\frac{\langle o(|k|),k\rangle}{|k|^2}
\endaligned$$
We now choose $h,k$ so that $k=ih$ and $h^2=\epsilon\,\overline{\partial\b(z)}$, and then let $\epsilon\to 0$. We get
\begin{equation}\label{motivation}
\limsup_{|h|=|k|\to 0}\left|\frac{\langle \b(z+h)-\b(z),\bar{h}\rangle}{|h|^2}-\frac{\langle \b(z+k)-\b(z),\bar{k}\rangle}{|k|^2}\right|\geq2|\partial \b(z)|,
\end{equation}
and therefore  $|\partial \b (z)|\leq \frac12\,\|\b\|_{\bar{Q}}$. If $\b$ is differentiable at every point $x$ the claim follows.
\end{proof}


\noindent
It is a well known fact that Zygmund functions admit a modulus of continuity of the form $\delta\,\log\frac1\delta$, but may fail to differentiable almost everywhere. Thus, removing the differentiability assumption in Lemma \ref{smoothbarQ} does not seem automatic. Our next goal consists of proving this is actually the case.

\begin{theo}\label{Qimpliesboundedpartial}
Let $\b:\R^2\to\R^2$ belong to the class $\bar{Q}$. Then, $\b$ is differentiable almost everywhere, has $BMO$ distributional derivatives, and $\partial \b \in L^\infty$. Moreover,
$$\|\partial \b \|_{L^\infty}\leq C\,\|\b\|_{\bar{Q}}$$
for some constant $C>0$.
\end{theo}

\begin{proof}
We first prove that it is not restrictive to assume that $\b$ has compact support. To do this, let us assume that the theorem is proved under the extra assumption that $\b$ has compact support.  Now, let us be given $\b\in\bar{Q}$ non compactly supported. and set $\b_t=g_t\,\b$, where
\begin{equation}\label{gt}
g_t(x)=\begin{cases}
1 & |x|\leq t\\
1-\frac1t\,\log\frac{\log |x|}{\log t}& t\leq |x|\leq t^{e^t}\\
0 & t^{e^t}\leq |x|.
\end{cases}
\end{equation}
Clearly, $\b_t$ has compact support. We now prove that $\b_t\in\bar{Q}$. For proving this, we denote $\tau_h\g(x)=\g(x+h)$ and $\Delta_h\g(x)=\tau_h\g(x)-\g(x)$. Then we observe that
$$
\aligned
\frac{\langle\Delta_h\b_t,\bar{h}\rangle}{|h|^2}&-\frac{\langle\Delta_k\b_t,\bar{k}\rangle}{|k|^2}
= \tau_hg_t\,\frac{\langle\Delta_h\b,\bar{h}\rangle}{|h|^2}-\tau_kg_t\,\frac{\langle\Delta_k\b,\bar{k}\rangle}{|k|^2}+\Delta_hg_t\,\frac{\langle \b,\bar{h}\rangle}{|h|^2}-\Delta_kg_t\,\frac{\langle \b,\bar{k}\rangle}{|k|^2}\\
&=(\tau_hg_t-\tau_kg_t)\,\frac{\langle \Delta_h\b,\bar{h}\rangle}{|h|^2}+\tau_kg_t\left(\frac{\langle\Delta_h\b,\bar{h}\rangle}{|h|^2}-\frac{\langle\Delta_k\b,\bar{k}\rangle}{|k|^2}\right)+\langle \b,\frac{\bar{h}\Delta_hg_t}{|h|^2}-\frac{\bar{h}\Delta_kg_t}{|k|^2}\rangle\\
&=(\Delta_hg_t-\Delta_kg_t)\,\frac{\langle \Delta_h\b,\bar{h}\rangle}{|h|^2}+\tau_kg_t\left(\frac{\langle\Delta_h\b,\bar{h}\rangle}{|h|^2}-\frac{\langle\Delta_k\b,\bar{k}\rangle}{|k|^2}\right)+\langle \b,\frac{\bar{h}\Delta_hg_t}{|h|^2}-\frac{\bar{h}\Delta_kg_t}{|k|^2}\rangle
\endaligned
$$
Now we use the Mean Value Theorem to deduce that
$$|\Delta_h g_t(x)|\leq \frac{C|h|}{t\,|x|\,\log|x|}\hspace{1cm}\text{and}\hspace{1cm}|\Delta_k g_t(x)|\leq \frac{C|k|}{t\,|x|\,\log|x|}$$
We now recall that $\b\in\bar{Q}$ implies that $\b\in Z$, and therefore $\b$ has $|x|\log|x|$ growth at infinity. Having in mind that $|g_t|\leq 1$, we have for $|h|=|k|$ that
$$
\left|
\frac{\langle\Delta_h\b_t,\bar{h}\rangle}{|h|^2}-\frac{\langle\Delta_k\b_t,\bar{k}\rangle}{|k|^2}\right|\leq
\left|
\frac{\langle\Delta_h\b,\bar{h}\rangle}{|h|^2}-\frac{\langle\Delta_k\b,\bar{k}\rangle}{|k|^2}\right|+\frac{C}{t}
$$
whence $g_t\b\in \bar{Q}$ and $\|g_t\b\|_{\bar{Q}}\leq \|\b\|_{\bar{Q}}+\frac{C}{t}$. We are now in situation to apply the theorem to $g_t\b$ and so $g_t\b$ is differentiable a.e. and moreover we have the bound
$$\|\overline\partial (g_t\b)\|_{L^\infty}\leq C\,\|g_t\b\|_{\bar{Q}}\leq C\|\b\|_{\bar{Q}}+\frac{C}{t}$$
The proof now finishes easily, as for any fixed $x$ one can always find $t>0$ large enough so that
$$\overline\partial \b(x)=g_t(x)\,\overline\partial \b(x)=\overline\partial (g_t\b)(x)-\b(x)\,\overline\partial g_t(x)$$
whence, after enlarging $t$ if needed, 
$$|\overline\partial \b(x)|\leq \|\overline\partial (g_t\b)\|_{L^\infty}+|\b(x)\,\overline\partial g_t(x)|\leq C\,\|g_t\b\|_{\bar{Q}}+\frac{C}{t}\leq C\|\b\|_{\bar{Q}}$$
as desired. Therefore, we can assume without loss of generality that $\b$ has compact support in $\R^2$.\\
\\
Through a dilation if needed, we will suppose that $\supp\b\subset\D$, where $\D$ denotes the unit disk on $\R^2$. Then, since $\b$ is continuous, the convolution $\u(z,y)=P_y\ast \b(z)$ is harmonic on $\R^2\times (0,+\infty)$ and continuous in $\R^2\times[0,+\infty)$. Also the complex derivative $\partial \u$ is harmonic in $\R^2\times (0,\infty)$, and as distributions one has 
$$\partial \u=\partial (P_y\ast \b)=\partial P_y\ast \b=P_y\ast \partial \b.$$
In particular, the last convolution is well defined, and from $\supp(\partial\b)\subset \D$ we have 
$$|\partial \u(z)| =|\partial P_y\ast\b(z)|\leq C\,\int_\D\frac{|\b(w)|}{|z-w|^3}\,dA(w)\leq \frac{C}{|z|^3}\hspace{1cm}\text{for each }z\notin2\D$$ 
uniformly for each $y>0$. In particular, $\partial \u\in L^p(\C\setminus 2\D)$ for each $\frac23<p<\infty$. From $\b\in \bar{Q}$ and $\|P_y\|_1=1$ we have that also $\u\in \bar{Q}$ and $\|\u\|_{\bar{Q}}\leq \|\b\|_{\bar{Q}}$, uniformly in $y>0$. So by Lemma \ref{smoothbarQ}, one has $2\|\partial \u\|_{L^\infty}\leq \| \u\|_{\bar{Q}}=\|\b\|_{\bar{Q}}$, and this uniformly in $y$.  It then follows that $\partial \u\in L^p(\C)$ uniformly in $y$, for each $1<p<\infty$. As an element of the harmonic Hardy space $h^p(\R^2\times (0,+\infty))$, $p>1$, we know that $\partial \u$ has well defined boundary values $\g\in L^p(\R^2)$, and moreover one necessarily has $\partial \u=P_y\ast \g$. Since also $P_y\ast \partial\b=P_y\ast \g$, and $p>1$, it then follows that $\partial\b=\g$ and so $\partial\b$ is actually an $L^p(\R^2)$ vector field.  By Lemma \ref{curldiv} we obtain $D\b\in L^p$. This already gives that $\b$ is differentiable a.e., because one can take any $p>2$ (see for instance \cite[Theorem 2.21]{Kinn}). Once we know that $\partial\b\in L^p$ and $P_y\ast \partial\b\in L^\infty$ we immediately infer that $\partial\b\in L^\infty$ with $\|\partial\b\|_{L^\infty}\leq \|P_y\ast \partial\b\|_{L^\infty}= \|\partial\u\|_{L^\infty}\leq \frac12\|\b \|_{\bar{Q}}$, and this with no dependence on $\supp\b$.  Using again Lemma \ref{curldiv} we get $D\b\in BMO$. In particular, $\b$ is differentiable almost everywhere. 
\end{proof}

\noindent
 In the converse direction, an extra assumption on the growth of $\b$ is needed.

\begin{theo}\label{partialfboundedimpliesQ}
Let $\b\in W^{1,1}_{loc}(\R^2;\R^2)$ be a continuous vector field such that 
\begin{equation}\label{xlogxgrowth}
\limsup_{|x|\to\infty}\frac{|\b(x)|}{|x|\,\log|x|}<\infty
\end{equation}
and that $\partial \b\in L^\infty$.  Then $\b\in\bar{Q}$ and $\|\b\|_{\bar{Q}}\leq C\,\|\partial \b\|_{L^\infty}$.
\end{theo}

\begin{proof}
This proof follows the scheme of \cite[Proposition 12]{Rei}. So we first assume that $\b$ has compact support. Fix two unit vectors $\alpha,\beta\in \R^2$, and set $a=\alpha h, b=\beta h$ for some $h>0$. For each vector field $\g:\R^2\to\R^2$, we define
$$\Delta \g(x)=\Delta_{a,b}\g(x)=\langle \g(x+a)-\g(x),\bar{\alpha}\rangle -\langle \g(x+b)-\g(x),\bar{\beta}\rangle.$$
Clearly, $\Delta=\Delta_{a,b}$ is a linear operator in $\g$, and 
\begin{equation}\label{elinftybound}
|\Delta \g(x)|\leq 4\,\|\g\|_{L^\infty}
\end{equation}
Moreover, $\g \in \bar{Q} $ if and only if $|\Delta \g|\leq C\,h$ for some constant $C$ that does not depend on $a$, $b$. We can represent $\Delta \g$ in terms of $\partial \g$ and $\overline\partial \g$ as follows, 
$$\aligned
\Delta \g(x)
&=\int_0^h \frac{d}{ds}\bigg(\langle \g(x+\alpha s), \bar\alpha\rangle-\langle \g(x+\beta s),\bar\beta\rangle\bigg)\,ds\\
&=\int_0^h  \langle D\g(x+\alpha s)\,\alpha, \bar\alpha\rangle-\langle D\g(x+\beta s)\,\beta,\bar\beta\rangle \,ds\\
&=\Re  \int_0^h   (\partial \g(x+\alpha s)\alpha^2- \partial \g(x+\beta s)\beta^2)\, ds
+\Re \int_0^h(\overline\partial \g(x+\alpha s) -\overline\partial \g(x+\beta s) ) \,ds\\&=\Delta_\partial \g(x)+\Delta_{\overline\partial}\g(x)
\endaligned$$
where we set
$$\aligned
\Delta_{\overline\partial}\g(x)&=\Re  \int_0^h   (\overline\partial \g(x+\alpha s)- \overline\partial \g(x+\beta s))\,ds\\
\Delta_\partial \g(x)&=\Re \int_0^h(\partial \g(x+\alpha s)\,\alpha^2 -\partial \g(x+\beta s)\,\beta^2 ) \,ds
\endaligned
$$
We now proceed with the proof. We denote $\u(x,y)=P_y\ast \b(x)$, $x\in\C$, $y\geq 0$. We know that $\u$ is harmonic in $\R^{3}_+$ and continuous up to the boundary, since $\b\in C_c(\C)$. For each $y>0$, 
$$
\aligned 
\b(x)=\u(x,0)
&=\int_0^yt\,\partial^2_{yy}\u(x,t)\,dt-y\,\partial_y\u(x,y)+\u(x,y)\\
&\equiv\int_0^y t\,\w_t(x)\,dt-y\,\v_y(x)+\u_y(x)
\endaligned$$
where we wrote $\u_y(x)=\u(x,y)$, $\v_y(x)= \partial_y\u(x,y)$ and $\w_r(x)= \partial_{yy}^2\u(x,r)$. By the linearity of $\Delta$, which acts only on the $x$ variable, one has
\begin{equation}\label{3terms}
\Delta \b(x)=\int_0^y t\, \Delta \w_t(x)\,dt-y\,\Delta \v_y(x)+\Delta \u_y(x).
\end{equation}
We now bound the three terms in the right hand side. For the first one, we use Lemma \ref{curldiv} to see that $\partial \b\in L^\infty$ implies $D\b\in BMO$, which in turn guarantees that $\b\in Z$. Now, from Lemma \ref{poissonZ} as well as equation \eqref{hessian1/y}, we deduce that $\|H \u\|_{L^\infty}\leq C\,\frac{ \|\b\|_Z}{y}$ which in turn gives us that 
$$\|\w_r\|_{L^\infty}\leq C\, \frac{\|\b\|_Z}{r}.$$
This fact, together with \eqref{elinftybound}, implies for the first term in \eqref{3terms} the bound
$$
\left|\int_0^y t\, \Delta \w_t(x)\,dt\right|\leq \int_0^y t\,4\|\w_t\|_{L^\infty}\,dt= C \,y\,\|\b\|_Z.
$$
For the second and third terms in \eqref{3terms}, we use that $\Delta =\Delta_{\overline\partial}+\Delta_\partial$,  
$$\aligned
y\,\Delta \v_y(x)&=y\,\Delta_{\overline\partial} \v_y(x)+y\,\Delta_\partial \v_y(x)\\
 \Delta \u_y(x)&= \Delta_{\overline\partial} \u_y(x)+ \Delta_\partial \u_y(x)\\
\endaligned
$$
and proceed first with the $\Delta_\partial$ terms. For each fixed $y$, Lemma \ref{easyLinfty} gives us that
$$ \aligned
\partial_{x_i}\,\u_y=  \partial_{x_i}  \,(P_y\ast \b)=P_y\ast (\partial_{x_i}\b)
&\Longrightarrow \partial \u_y=P_y\ast \partial\b\\
&\Longrightarrow \|\partial\u_y\|_{L^\infty}=\|P_y\ast \partial\b\|_{L^\infty}\leq \|\partial\b\|_{L^\infty}\endaligned$$
On the other hand, since $u$ is smooth, we can argue similarly to get that
$$\aligned
\partial_{x_i}\v_y = \partial^2_{y, x_i}\u =  \partial_y\,\left(P_y\ast \partial_{x_i}\b\right)
&\Longrightarrow \partial\v_y = \partial_y(P_y\ast \partial\b)\\
&\Longrightarrow \|\partial\v_y\|_{L^\infty} =\|\partial_y(P_y\ast \partial\b)\|_{L^\infty} \leq C \,\frac{\|\partial\b\|_{L^\infty}}{y}.
\endaligned
$$
Thus, from $|\Delta_\partial \g(x)|\leq 2h\,\|\partial \g\|_{L^\infty}$ one gets that
$$
\aligned
|\Delta_\partial\u_y(x)|&\leq 2h\,\|\partial\u_y\|_{L^\infty} \leq C \,h\,\|\partial\b\|_{L^\infty},\\
|y\,\Delta_\partial\v_y(x)|&\leq 2h y \|\partial\v_y\|_{L^\infty}\leq C \,h\,\|\partial\b\|_{L^\infty}.
\endaligned
$$
Now we proceed with the $\Delta_{\overline\partial}$ terms. Calling $\gamma=\frac{\alpha-\beta}{|\alpha-\beta|}$, we see that
$$ \aligned
|\Delta_{\overline\partial}\g(x)|
&=\left|\Re  \int_0^h   \int_0^{s|\alpha-\beta|}\frac{d}{d\sigma}(\overline\partial \g(x+\beta s+\gamma\sigma) )\,d\sigma\,ds\right|\\
&=\left|\Re\int_0^h   \int_0^{s|\alpha-\beta|} D(\overline\partial \g(x+\beta s+\gamma\sigma)\cdot\gamma  )\,d\sigma\,ds\right|
\leq   \frac{h^2\,|\alpha-\beta|}{2}\, \|D(\overline\partial \g)\|_{L^\infty}
\endaligned$$
After applying this to $\g=\u_y$ and to $\g=\v_y$, and putting al together in \eqref{3terms}, one obtains 
\begin{equation}\label{finalstep1}
|\Delta\b(x)|\leq C \,y\,\|\b\|_Z+C \,h\,\|\partial\b\|_{L^\infty}+\frac{h^2\,|\alpha-\beta| }2\left(\| D(\overline\partial\u_y)\|_{L^\infty}+y\| D(\overline\partial\v_y) \|_{L^\infty}\right)
\end{equation}
Lemma \ref{curldiv} tells that from $\partial\b\in L^\infty$ we get $\overline\partial\b\in BMO$ and so $P_y\ast(\overline\partial\b)$ is harmonic Bloch. This, together with Lemma \ref{poissonBMO}, implies that 
$$
\aligned
\u_y=P_y\ast\b \hspace{.5cm}
&\Longrightarrow\hspace{.5cm}\overline\partial\u_y= P_y\ast \overline\partial\b\\
&\Longrightarrow\hspace{.5cm}D(\overline\partial\u_y)= D(P_y\ast \overline\partial\b)\\
&\Longrightarrow \hspace{.5cm}\|D(\overline\partial\u_y)\|_{L^\infty}=\|D(P_y\ast \overline\partial\b)\|_{L^\infty}\leq C \,\frac{\|\overline\partial\b\|_\ast}{y}\leq C \,\frac{\|\partial\b\|_{L^\infty}}{y}.
\endaligned$$
Similarly,
$$
\aligned
\v_y=\partial_y\u_y=\partial_yP_y\ast\b
&\hspace{.5cm}\Longrightarrow\hspace{.5cm}\overline\partial\v_y= \partial_yP_y\ast \overline\partial\b\\
&\hspace{.5cm}\Longrightarrow\hspace{.5cm}D(\overline\partial\v_y)= D(\partial_yP_y\ast \overline\partial\b)\\
&\hspace{.5cm}\Longrightarrow\hspace{.5cm} \|D(\overline\partial\v_y)\|_{L^\infty}=\|D(\partial_yP_y\ast \overline\partial\b)\|_{L^\infty}\leq C \,\frac{\|\overline\partial\b\|_\ast}{y^2}\leq C\, \frac{\|\partial\b\|_{L^\infty}}{y^2}
\endaligned$$
Plugging the above bounds into \eqref{finalstep1}, we get
$$
|\Delta\b(x)|
\leq C \,y\,\|\b\|_Z+C \,h\,\|\partial\b\|_{L^\infty}+
C \,\frac{h^2\,|\alpha-\beta| }{2y}
$$
and choose $y=h$ to get $|\Delta\b(x)|\leq C \,h\,\|\partial\b\|_{L^\infty}$. So $\b\in \bar{Q}$ and $\|\b\|_{\bar{Q}}\leq C\,\|\partial\b\|_{L^\infty}$. The claim follows in the case $\b\in C_c(\C)$. \\
\\
In order to remove the assumption on the compact support, we use again Reimann's ideas. So we use the $g_t$ functions introduced at \eqref{gt}, and assume that $\partial \b\in L^\infty$ and $|\b(x)|\leq C\,|x|\,\log|x|$ as $|x|\to\infty$. For every fixed $t>0$, we have that $\partial (g_t\b)=\b\,\partial g_t  + g_t\,\partial \b$ and so $\partial(g_t\b)\in L^\infty$. Moreover, $g_t\b$ has compact support. It then follows that $g_t\b\in \bar{Q}$ and $\|g_t\b\|_{\bar{Q}}\leq C\,\|\partial(g_t\b )\|_{L^\infty}$. However, from \eqref{xlogxgrowth} we see that
$$\aligned
|\partial (g_t\b)(x)|
&\leq  |\partial \b(x)| + |\b(x)||\partial   g_t(x)|\\
&\leq |\partial \b(x)|+ C\,|x|\log|x|\frac{1}{t|x|\,\log|x|}\\
&\leq |\partial \b(x)| + \frac{C}t
\endaligned$$
Thus, we can always pick $t>0$ large enough so that $\|g_t\b\|_{\bar{Q}}\leq C\|\partial \b\|_{L^\infty}$. We now fix $x\in\R^2$. For every pair $|h|=|k|$ there is always $t>0$ large enough and such that $|x|,|x+h|,|x+k|< t^{e^t}$ so that $\b=g_t\b$ at $x$, $x+h$ and $x+k$. Thus, when evaluating the $\bar{Q}$ norm of $\b$ at $x, x+h$ and $x+k$ one reduces the differences of $\b$ to the differences of $g_t\b$, which are controlled by $\|g_t\b\|_{\bar{Q}}$, which is independent of $t$, $|h|$ and $|k|$. It follows that $\b\in \bar{Q}$ and $\|\b\|_{\bar{Q}}\leq C\,\|\partial \b\|_{L^\infty}$.

\end{proof}

\noindent
In the above proof, among all terms in the right hand side of \eqref{3terms}, most of them admit the desired key bound precisely because $\b\in Z$, except the two $\Delta_\partial$ terms, which are the only ones requiring specifically that $\partial\b\in L^\infty$.  \\
\\
On the other hand, one can deduce from the previous Theorem that $\bar{Q}$ contains many non-trivial, non-Lipschitz vector fields. At least, as many as non-Lipschitz solutions of the planar Euler system with bounded vorticity.

\begin{coro}\label{Qnonlipschitz}
$\bar{Q}$ contains many non-Lipschitzian vector fields.
\end{coro}
\begin{proof}
Let us assume that $\omega_0:\C\to \R$ is a real valued, compactly supported function, such that $\omega_0\in L^\infty$. It follows from Yudovich Theorem \cite{Y} that the associate Euler system, in its vorticity form
$$
\begin{cases}
\omega_t+(\v\cdot\nabla)\omega=0\\
\v(t,\cdot)=\frac{1}{2\pi}\,\frac{(y,-x)}{x^2+y^2}\ast\omega(t,\cdot)\\
\omega(0,\cdot)=\omega_0\end{cases}
$$ 
admits a unique solution $\omega$ global in time, belonging to $L^\infty((0,\infty); L^\infty(\C))$, and whose associate velocity field $v$ is such that $\curl\v=\omega$, that is, $2 \partial\v = i\omega$. In particular, $\partial\v(t,\cdot)\in L^\infty$ for every $t$. Therefore, $\v(t,\cdot)$ is an element of $\bar{Q}$ at every time. However, it is well known that not all bounded vorticities produce Lipschitz vector fields, see for instance the example by Bahouri and Chemin in  \cite[Theorem 1.3]{Chex}.
\end{proof}

\section{An alternative to $\bar{Q}$: the class $R$}\label{Rsection}

The class $\bar{Q}$ is an appropiate counterpart to Reimann's $Q$ class when $n=2$, but seems not so convenient if $n>2$ due to the absence of complex conjugation. The following observation shows that there is another way to recover $|\partial \b(x)|$ from the Taylor development of $\b$ at $x$ that may be more convenient with higher dimensional counterparts. 

\begin{lem}\label{boundedcurldiv}
Let $\b$ be a vector field in $\R^2$. Assume that $x$ is a differentiability point of $\b$. Then
$$
\limsup_{|h|,|k|\to 0} \sup_{0\leq\theta\leq 2\pi}\frac{|\langle \b(x+h)-\b(x), e^{i\theta} k\rangle-\langle \b(x+k)-\b(x), e^{i\theta} h\rangle |}{|h|\,|k|} = 2\left| \partial \b(x) \right|.
$$
\end{lem}
\begin{proof}
We first note that
$$\aligned
\langle D\b(x)h, e^{i\theta} k\rangle-\langle D\b(x)k, e^{i\theta} h\rangle
&=\Re\bigg( (\partial \b(x) h+\overline\partial \b(x)\,\bar{h})e^{-i\theta}\bar{k}-(\partial \b(x) k+\overline\partial \b(x)\,\bar{k})e^{-i\theta}\bar{h}\bigg)\\
&=\Re\bigg( \partial\b(x)e^{-i\theta}( h  \bar{k}- k \bar{h})\bigg)\\
&=-2\Im\bigg( \partial\b(x)e^{-i\theta}\bigg)\,\Im( h  \bar{k} )\\
&= \bigg(-2 \Im(\partial\b(x))\,\cos\theta +2\,\Re(\partial\b(x))\,\sin\theta\bigg)\,\Im( h  \bar{k} )
\endaligned$$
But since $\b$ is differentiable at $x$ we know that
$$
\limsup_{|h|\to 0}\frac{|\langle\b(x+h)-\b(x)-D\b(x)h, e^{i\theta}k\rangle|}{|h|\,|k|}=\limsup_{|k|\to 0}\frac{|\langle\b(x+k)-\b(x)-D\b(x)k, e^{i\theta}h\rangle|}{|h|\,|k|}=0
$$
Thus
$$\aligned
2\Im\bigg( \partial\b(x)e^{-i\theta}\bigg)\,\frac{\Im( h  \bar{k} )}{|h|\,|k|}
&=-\frac{\langle D\b(x)h, e^{i\theta} k\rangle-\langle D\b(x)k, e^{i\theta} h\rangle}{|h|\,|k|}\\
&=-\frac{\langle\b(x+h)-\b(x), e^{i\theta} k\rangle-\langle\b(x+k)-\b(x), e^{i\theta} h\rangle}{|h|\,|k|}\\
&+\frac{\langle o(h), k\rangle}{|h|\,|k|}+\frac{\langle o(k), h\rangle}{|h|\,|k|}
\endaligned$$
so it is obvious that if we take first supremum in $\theta$ and then $\limsup$ in $h,k$ one gets
$$\limsup_{h,k\to 0}\sup_\theta\left|\frac{\langle\b(x+h)-\b(x), e^{i\theta} k\rangle-\langle\b(x+k)-\b(x), e^{i\theta} h\rangle}{|h|\,|k|}\right|\leq \left|2 \partial\b(x) \right|.$$
For the converse inequality, just choose $k=ih$, then take supremum in $\theta$ and let $h\to 0$ then
$$
\left|2 \partial\b(x) \right| \leq \limsup_{h,k\to 0}\left|\frac{\langle\b(x+h)-\b(x), e^{i\theta} k\rangle-\langle\b(x+k)-\b(x), e^{i\theta} h\rangle}{|h|\,|k|}\right|.
$$
The claim follows. 
\end{proof}

\noindent
Lemma \ref{boundedcurldiv} encourages us to introduce the following definition. 

\begin{defi}\label{definitionRn=2}
We say that a continuous function $\b:\R^2\to\R^2$ is an element of the class $R$ if
$$
\sup_{x\in\R^2}\sup_{|h|=|k|\neq0}\sup_{0\leq\theta\leq 2\pi}\frac{|\langle\b(x+h)-\b(x), e^{i\theta} k\rangle-\langle\b(x+k)-\b(x), e^{i\theta} h\rangle |}{|h|\,|k|} \leq C.
$$
The best constant $C$ will be denoted by $\|\b\|_{R}$.
\end{defi}

\noindent
It is not hard to see that we have the inequalities
$$\|\b\|_{Z}\leq c\,\|\b\|_R\leq c\,\|\b\|_{Lip}.$$
As it was for $\bar{Q}$, these inequalities are actually a direct consequence of Propositon \ref{zygmund}, which will be proven in the next sections. Also, it is not hard to deduce from Lemma \ref{boundedcurldiv} that if $\b\in R$ happens to be smooth then one has the bound
\begin{equation}\label{Rsmooth}
\|\partial \b\|_{L^\infty}\leq \frac12\,\|\b\|_R,
\end{equation}
arguing as we did in Lemma \ref{smoothbarQ}. As in the previous section, the difficulty is in proving that \eqref{Rsmooth} also holds true in absence of smoothness. 

\begin{theo}\label{Rimpliesboundeddb}
Let $\b:\R^2\to\R^2$ belong to the class $R$. Then $\b$ is differentiable almost everywhere, it has $BMO$ distributional derivatives, and $\partial \b\in L^\infty$ with $\|\partial \b\|_{L^\infty}\leq\frac12\|\b\|_{R}$.
\end{theo}
\begin{proof}
The proof of the above result follows the lines of the proof we have given in Theorem \ref{Qimpliesboundedpartial}, so we omit it. 
\end{proof}
\noindent
The above sufficient condition for belonging to $R$ is also necessary, again with the growth condition.

\begin{theo}\label{partialfboundedimpliesR}
Let $\b\in W^{1,1}_{loc}(\R^2;\R^2)$ be a vector field such that 
$$\limsup_{|x|\to\infty}\frac{|\b(x)|}{|x|\,\log|x|}<\infty$$
and that $\partial \b\in L^\infty$.  Then $\b\in R$ and $\|\b\|_R\leq C\,\|\partial \b\|_{L^\infty}$.
\end{theo}
\begin{proof}
Even though he proof is similar to the proof of Theorem \ref{partialfboundedimpliesQ}, some modifications need to be done. As before, we only do it assuming that $\b$ has compact support (removing this assumption can be done as in Theorem \ref{partialfboundedimpliesQ}), and start by fixing two unit vectors $\alpha,\beta\in \R^2$, and set $a=\alpha h, b=\beta h$ for some $h>0$. Given $\g:\R^2\to\R^2$, this time one sets
$$\Delta \g(x)=\Delta_{a,b,\theta}\g(x)=\langle \g(x+a)-\g(x),e^{i\theta}\beta\rangle -\langle \g(x+b)-\g(x),e^{i\theta}\alpha\rangle.$$
Here, $\theta\in\{0,\pi/2\}$. Clearly, $\Delta=\Delta_{a,b,\theta}$ is a linear operator in $\g$, and 
\begin{equation}\label{Linftybound}
|\Delta \g(x)|\leq 4\,\|\g\|_{L^\infty}
\end{equation}
Moreover, $\g \in R$ if and only if $|\Delta \g|\leq C\,h$ for some constant $C$ that does not depend on $a$, $b$ or $\theta$. The representation of $\Delta \g$ in terms of $\partial \g$ and $\overline\partial \g$ changes a bit with respect to that in Theorem \ref{partialfboundedimpliesQ}, 
$$\aligned
\Delta \g(x)
&=\int_0^h \frac{d}{ds}\bigg(\langle \g(x+\alpha s),e^{i\theta}\beta\rangle-\langle \g(x+\beta s),e^{i\theta}\alpha\rangle\bigg)\,ds\\
&=\int_0^h  \langle D\g(x+\alpha s)\,\alpha, e^{i\theta}\beta\rangle-\langle D\g(x+\beta s)\,\beta,e^{i\theta}\alpha\rangle \,ds\\
&=\Delta_{\overline\partial}\g(x)+\Delta_{ \partial}\g(x),
\endaligned$$
where we have set
$$\aligned
\Delta_{\overline\partial}\g(x)&=\Re \left(e^{-i\theta} \int_0^h   (\overline\partial \g(x+\alpha s)- \overline\partial \g(x+\beta s))\,\bar{\beta}\bar{ \alpha}\,ds\right),\\
\Delta_\partial \g(x)&=\Re\left( e^{-i\theta}\int_0^h(\partial \g(x+\alpha s)\,\alpha\bar{\beta}-\partial \g(x+\beta s)\,\beta\bar{\alpha} ) \,ds\right).\endaligned
$$
The proof now follows as the one of Theorem \ref{partialfboundedimpliesQ}. So for $\u(x,y)=P_y\ast \b(x)$ one knows that $\u$ is harmonic in $\R^{3}_+$ and continuous up to the boundary, since $\b\in C_c(\C)$. For each $t>0$, 
$$
\aligned 
\b(x)=\u(x,0)
&=\int_0^yt\,\partial^2_{yy}\u(x,t)\,dt-y\,\partial_y\u(x,y)+\u(x,y)\\
&\equiv\int_0^y t\,\w_t(x)\,dt-y\,\v_y(x)+\u_y(x)
\endaligned$$
where we wrote $\u_y(x)=\u(x,y)$, $\v_y(x)= \partial_y\u(x,y)$ and $\w_r(x)= \partial_{yy}^2\u(x,r)$. By the linearity of $\Delta$, which acts only on the $x$ variable, one has
$$
\Delta \b(x)=\int_0^y t\, \Delta \w_t(x)\,dt-y\,\Delta\v_y(x)+\Delta \u_y(x).
$$
and now one proceeds term by term. For the $\w$ term, one can use Lemma \ref{poissonZ} to see that
$$\aligned
\partial \b\in L^\infty\hspace{1cm}
&\Longrightarrow\hspace{1cm}D\b\in BMO\\
&\Longrightarrow\hspace{1cm} \b\in Z\\
&\Longleftrightarrow\hspace{1cm} \|H \u\|_{L^\infty}\leq C\,\frac{ \|\b\|_Z}{y}\hspace{1cm}\Rightarrow \|\w_r\|_{L^\infty}\leq C\, \frac{\|\b\|_Z}{r}\endaligned$$
Hence
$$
\left|\int_0^y t\, \Delta \w_t(x)\,dt\right|\leq \int_0^y t\,4\|\w_t\|_{L^\infty}\,dt= C \,y\,\|\b\|_Z
$$
As desired. For the other two terms, we use that $\Delta =\Delta_{\overline\partial}+\Delta_\partial$,  
$$\aligned
y\,\Delta \v_y(x)&=y\,\Delta_{\overline\partial} \v_y(x)+y\,\Delta_\partial \v_y(x)\\
 \Delta \u_y(x)&= \Delta_{\overline\partial} \u_y(x)+ \Delta_\partial \u_y(x)\\
\endaligned
$$
and proceed first with the $\Delta_\partial$ terms. For each fixed $y$, Lemma \ref{easyLinfty} gives us that
$$ \aligned
\partial_{x_i}\,\u_y=  \partial_{x_i}  \,(P_y\ast \b)=P_y\ast (\partial_{x_i}\b)
&\Longrightarrow \partial \u_y=P_y\ast \partial \b\\
&\Longrightarrow \|\partial \u_y\|_{L^\infty}=\|P_y\ast \partial \b\|_{L^\infty}\leq C \,\|\partial \b\|_{L^\infty}\endaligned$$
On the other hand, since $\u$ is smooth, we can argue similarly to get that
$$\aligned
\partial_{x_i} \v_y = \partial^2_{y, x_i} \u =  \partial_y\,\left(P_y\ast \partial_{x_i} \b\right)
&\Longrightarrow \partial \v_y = \partial_y(P_y\ast \partial \b)\\
&\Longrightarrow \|\partial \v_y\|_{L^\infty} =\|\partial_y(P_y\ast \partial \b)\|_{L^\infty} \leq C \,\frac{\|\partial \b\|_{L^\infty}}{y}.
\endaligned
$$
Thus
$$
\aligned
|\Delta_\partial \u_y(x)|&\leq 2h\,\|\partial \u_y\|_{L^\infty} \leq C \,h\,\|\partial \b\|_{L^\infty}\\
|y\,\Delta_\partial \v_y(x)|&\leq 2h y \|\partial \v_y\|_{L^\infty}\leq C \,h\,\|\partial \b\|_{L^\infty}
\endaligned
$$
where $C$ is a constant. Concerning the $\Delta_{\overline\partial}$ terms, we call $\gamma=\frac{\alpha-\beta}{|\alpha-\beta|}$, and observe that
$$ \aligned
|\Delta_{\overline\partial}\g(x)|
&=\left|\Re \left(e^{-i\theta} \int_0^h   \int_0^{s|\alpha-\beta|}\frac{d}{d\sigma}(\overline\partial \g(x+\beta s+\gamma\sigma) )\,\bar{\beta}\bar{ \alpha}\,d\sigma\,ds\right)\right|\\
&=\left|\Re \left(e^{-i\theta} \int_0^h   \int_0^{s|\alpha-\beta|} D(\overline\partial \g(x+\beta s+\gamma\sigma)\cdot\gamma  )\,\bar{\beta}\bar{ \alpha}\,d\sigma\,ds\right)\right|
\leq   \frac{h^2\,|\alpha-\beta|}{2}\, \|D(\overline\partial \g)\|_{L^\infty}
\endaligned$$
After applying this to $\g=\u_y$ and to $\g=\v_y$, one obtains 
\begin{equation}\label{finalstep}
|\Delta \b(x)|\leq C \,y\,\|\b\|_Z+C \,h\,\|\partial \b\|_{L^\infty}+\frac{h^2\,|\alpha-\beta| }2\left(\| D(\overline\partial \u_y)\|_{L^\infty}+y\| D(\overline\partial \v_y) \|_{L^\infty}\right)
\end{equation}
We now use the first inequality in Lemma \ref{poissonBMO} with $\g= \overline\partial \u_y$. Indeed, by the linearity of all the involved operators
$$
\aligned
\u_y=P_y\ast \b \hspace{1cm}
&\Rightarrow\hspace{1cm}\overline\partial \u_y= P_y\ast \overline\partial \b
\endaligned$$
Now, since $\partial \b\in L^\infty$ we have $\overline\partial \b\in BMO$ and therefore $\b\in Z$, so Lemma \ref{poissonZ} applies,
$$
\|D(\overline\partial \u_y)\|_{L^\infty}=\|D(P_y\ast \overline\partial \b)\|_{L^\infty}\leq C \,\frac{\| \b\|_Z}{y}.$$
For $\g=\overline\partial \v_y$, we proceed similarly, and observe that  
$$
\aligned
\v_y=\partial_y\u_y=\partial_yP_y\ast \b\hspace{1cm}
&\Rightarrow\hspace{1cm}\overline\partial \v_y= \partial_yP_y\ast \overline\partial \b.
\endaligned$$
Hence, one may combine Lemmas \ref{higherorder} and \ref{poissonZ} to get
$$
\|D(\overline\partial \v_y)\|_{L^\infty}=\|D(\partial_yP_y\ast \overline\partial \b)\|_{L^\infty}\leq C \,\frac{\|\b\|_Z}{y^2}.$$
We now plug the above bounds into \eqref{finalstep},
$$
|\Delta \b(x)|
\leq C \,y\,\|\b\|_Z+C \,h\,\|\partial \b\|_{L^\infty}+
C \,\frac{h^2\,|\alpha-\beta| }{2y}
$$
and choose $y=h$ to get $|\Delta \b(x)|\leq C \,h\,\|\partial \b\|_{L^\infty}$. So $\b\in R$ and $\|\b\|_{R}\leq C\,\|\partial \b\|_{L^\infty}$. The claim follows in the case $\b\in C_c(\C)$. \end{proof}

\noindent
The following corollary, Theorem A in the introduction, is a way of putting together Theorems \ref{Qimpliesboundedpartial}, \ref{partialfboundedimpliesQ}, \ref{Rimpliesboundeddb} and \ref{partialfboundedimpliesR}.

\begin{coro}
Let $\b:\R^2\to\R^2$ be a continuous vector field. The following conditions are equivalent:
\begin{itemize}
\item[(a)] $\b\in R$
\item[(b)] $\b\in \bar{Q}$
\item[(c)] $\b$ is differentiable a.e., $\partial \b\in L^\infty$  and $|\b(x)|\leq C|x|\,\log|x|$ as $|x|\to\infty$.
\end{itemize}
Moreover, in case this happens, then $\|\b\|_{\bar{Q}}\simeq\|\b\|_R\simeq \|\partial \b\|_{L^\infty}\simeq \|\div\b\|_{L^\infty}+\|\curl\b\|_{L^\infty}$.
\end{coro}

\noindent
As explained at the beginning of this section, the absence of complex conjugation in $\R^n$ when $n>2$ seems to make the $R$ class more suitable for higher dimensional counterparts. In order to build them, one may replace the rotation factor $e^{i\theta}$ in Definition \ref{definitionRn=2} by rotations not only in the $Ox_1,x_2$ plane, but on any of the coordinate planes $Ox_i,x_j$. For this, let us introduce the set ${\mathcal J}_n=\{J_{i,j}\}_{1\leq i<j\leq n}$ of matrices $J_{i,j}\in\R^{n\times n}$ defined by 
$$\aligned J_{i,j}e_i&=-e_j\\ J_{i,j}e_j &= e_i\\J_{i,j}e_k&=e_k\,,\,\,k\neq i,j\endaligned$$
where $e_1,\dots,e_n$ is the canonical basis in $\R^n$. When $n=2$, ${\mathcal J}_n$ contains only the matrix
$$\left(\begin{array}{cc}0&1\\-1&0\end{array}\right)$$
which is nothing but the rotation $e^{-i\frac{\pi}2}$. More in general, ${\mathcal J}_n$ contains $\frac{n(n-1)}{2}$ elements. 

\begin{lem}\label{Rnfailure}
Suppose that $n\geq 3$. Let $\b:\R^n\to\R^n$ be a vector field, and assume that $x$ is a differentiability point. If
\begin{equation}\label{Rnattempt}
\limsup_{|h|=|k|\to 0}\sup_{J\in {\mathcal J}_n\cup\{\Id\}} \left|\frac{\langle \b(x+h)-\b(x),Jk\rangle}{|h|\,|k|}-\frac{\langle \b(x+k)-\b(x), Jh\rangle}{|h|\,|k|}\right|\leq C_0
\end{equation}
then also $|D\b(x)|\leq C\,C_0$ for some dimensional constant $C$. 
\end{lem}

\begin{proof}
Since $x$ is a differentiability point, 
$$
\aligned
\limsup_{|h|=|k|\to 0}&\left|\frac{\langle \b(x+h)-\b(x),Jk\rangle}{|h|\,|k|}-\frac{\langle \b(x+k)-\b(x), Jh\rangle}{|h|\,|k|}\right|\\
&=\limsup_{|h|=|k|\to 0} \left|\frac{\langle D\b(x)h, Jk\rangle -\langle D\b(x)k, Jh\rangle}{|h|\,|k|}\right|\\
&=\sup_{|h|=|k|=1} |\langle D\b(x)h, Jk\rangle -\langle D\b(x)k, Jh\rangle|=\sup_{|h|=|k|=1} |\langle h, (D^t\b(x)J -J^t D\b(x))k \rangle|\\
\endaligned$$
When taking $J=\Id$ one recovers the curl matrix $D\b(x)-D^t\b(x)$, 
$$
\langle D\b(x)h, Jk\rangle -\langle D\b(x)k, Jh\rangle = \langle h, (D^t\b(x) J-J^tD\b(x))k\rangle=\langle h, (D^t\b(x)-D\b(x))k\rangle
$$
Let us now take $J=J_{i,j}$ for a given pair $1\leq i<j\leq n$. We get
$$
\aligned
\langle D\b(x)e_i, Je_j\rangle -\langle D\b(x)e_j, Je_i\rangle &= \langle \partial_i \b, e_i\rangle + \langle \partial_j \b, e_j\rangle= \partial_i\b_i+\partial_j\b_j\\
\langle D\b(x)e_i, Je_k\rangle -\langle D\b(x)e_k, Je_i\rangle &=\langle \partial_i \b, e_k\rangle+\langle \partial_k \b, e_j\rangle=\partial_i\b_k+\partial_k\b_j,\hspace{1cm}k\neq i,j\\
\langle D\b(x)e_j, Je_k\rangle -\langle D\b(x)e_k, Je_j\rangle &=\langle \partial_j \b, e_k\rangle-\langle \partial_k \b, e_i\rangle=\partial_j\b_k-\partial_k\b_i,\hspace{1cm}k\neq i,j
\endaligned$$
Suming up the second quantity with $-\partial_i\b_k+\partial_k\b_i$, and the third with $-\partial_j\b_k+\partial_k\b_j$ (both of which come from $D\b-D^t\b$), we get that both $\partial_k\b_j+\partial_k\b_i$ and $-\partial_k\b_i +\partial_k\b_j$ are bounded by multiples of $C_0$, which means that $\partial_k\b_i,\partial_k\b_j$ are bounded by multiples of $C_0$ whenever $k\neq i,j$. Moving now $i, j$ we obtain the same sort of boundedness for all non-diagonal elements of $Db$. Also, note that the boundedness of all pairs $\partial_i\b_i+\partial_j\b_j$ implies that of all diagonal elements. This finishes the proof. 
\end{proof}

\noindent
The above result shows that the class of vector fields $\b:\R^n\to\R^n$ satisfying \eqref{Rnattempt} reduces, when $n>2$, to Lipschitz vector fields. In contrast, when $n=2$, this class is much larger: this can be deduced from Lemma \ref{Qnonlipschitz}, together with the fact that in the plane one has $\bar{Q}=R$. This suggests it is not a good idea to build higher dimensional counterparts to $R$ in this way, because the class of vector fields one obtains is included into the Lipschitz ones, which are well understood. 

\section{Extending to higher dimensions: the class $R_0$}\label{R^nsection}

\noindent
Lemma \ref{boundedcurldiv} gives light to another fact: one may separate $\curl\b$ from $\div\b$ by simply choosing different values of $\theta$. This is the starting point to our following observation. Let us fix an integer $n\geq 2$. 

\begin{lem}\label{pointwisecurl}
Let $\b:\R^n\to\R^n$ be a vector field, and assume that $x$ is a differentiability point of $\b$. Then
$$
\limsup_{|h|,|k|\to 0} \frac{|\langle \b(x+h)-\b(x), k\rangle-\langle \b(x+k)-\b(x), h\rangle |}{|h|\,|k|}= |D\b(x)-D^t\b(x)|.
$$
\end{lem}
\begin{proof}
We first observe that if $\b$ is differentiable at $x$, then
$$\aligned
\langle \b(x+h)-\b(x), k\rangle 
&= \langle D\b(x)h, k\rangle+ \langle o(|h|),k\rangle \\
\endaligned$$
Now, after exchanging the roles of $h$ and $k$, we also have
$$\aligned
\langle \b(x+k)-\b(x), h\rangle 
&= \langle D\b(x) k, h\rangle+\langle o(|k|), h\rangle \\
\endaligned$$
Thus
$$
\frac{\langle \b(x+h)-\b(x), k\rangle-\langle \b(x+k)-\b(x), h\rangle }{|h|\,|k|}=
\frac{ \langle (D\b(x)-D^t\b(x))\, h, k\rangle }{|h|\,|k|}+\frac{\langle o(|h|),k\rangle}{|h|\,|k|}-\frac{\langle o(|k|),h\rangle}{|k|\,|h|}.
$$
and therefore one immediately gets
$$\limsup_{|h|,|k|\to 0}\frac{|\langle \b(x+h)-\b(x), k\rangle-\langle \b(x+k)-\b(x), h\rangle| }{|h|\,|k|}\leq |D\b(x)-D^t\b(x)|.$$
For the converse inequality, we recall that
$$|D\b(x)-D^t\b(x)|=\sup_{h,k\neq 0}\frac{ \langle (D\b(x)-D^t\b(x))\, h, k\rangle }{|h|\,|k|}$$
so we can pick two sequences $h_m, k_m\to 0$ such that
$$\aligned
|D\b(x)-D^t\b(x)|&=\lim_{m\to\infty}\frac{ \langle (D\b(x)-D^t\b(x))\, h_m, k_m\rangle }{|h_m|\,|k_m|}\\
&=\lim_{m\to\infty}\frac{|\langle \b(x+h_m)-\b(x), k_m\rangle-\langle \b(x+k_m)-\b(x), h_m\rangle |}{|h_m|\,|k_m|}\\
&\leq\limsup_{|h|,|k|\to 0}\frac{|\langle \b(x+h)-\b(x), k\rangle-\langle \b(x+k)-\b(x), h\rangle |}{|h|\,|k|}
\endaligned$$
and the claim follows.
\end{proof}

\noindent
The above result motivates the following definition.

\begin{defi}
We say that a continuous function $\b:\R^n\to\R^n$ belongs to the class $R_0$ if 
$$\frac{\left|\langle \b(x+h)-\b(x), k\rangle-\langle \b(x+k)-\b(x), h\rangle\right|}{|h|\,|k|}\leq C $$
for each pair $h,k$ such that $|h|=|k|\neq 0$. The best constant $C$ will be denoted as $\|\b\|_{R_0}$.
\end{defi}

\noindent
When $n=1$, the only options are $h=k$ (which gives nothing) or $h=-k$, which gives us that in fact $R_0$ is exactly the class of Zygmund functions. When $n=2$, though, the above definition suggests that $R_0$ is larger than $R$. For a general $n>1$, taking $k=-h$ in the definition we obtain that
$$\frac{\left|\langle \b(x+h)+ \b(x-h)-2\b(x), h\rangle\right|}{|h|^2}\leq C $$
which suggests that there may be some connection between $R_0$ and $Z$.
 
\begin{prop}\label{zygmund}
One has $R_0\subset Z$, and moreover $\|\b\|_Z\leq 4\,\|\b\|_{R_0}$. 
\end{prop}
\begin{proof}
Let us assume for a while that $a,b\in\R^n$ are such that  $\langle a, b\rangle=0$. Then by Pitagoras $|a+b|=|a-b|$ and thus we can use condition $R_0$ for $x=z+a$, $h=b-a$ and $k=-b-a$. We get
$$\aligned
|\langle \b(z+b)-\b(z+a), -b-a \rangle &- \langle \b(z-b) -\b(z+a), b-a\rangle |\\
&= |\langle \b(z-b)-\b(z+b), a\rangle - \langle \b(z+b)+\b(z-b), b\rangle + 2\langle \b(z+a), b\rangle|\\
&\leq \|\b\|_{R_0}\,|-b-a|\,|b-a| =\|\b\|_{R_0}\,(|a|^2+|b|^2)\endaligned$$
Similarly, for $x=z-a$, $h=b+a$ and $k=-b+a$, 
$$\aligned
|\langle \b(z+b)-\b(z-a), -b+a \rangle &- \langle \b(z-b) -\b(z-a), b+a\rangle| \\
&= |-\langle \b(z-b)-\b(z+b), a\rangle - \langle \b(z+b)+\b(z-b), b\rangle + 2\langle \b(z-a), b\rangle|\\
&\leq \|\b\|_{R_0} \,|-b+a|\,|b+a|= \|\b\|_{R_0}\,(|a|^2+|b|^2)
\endaligned$$
Summing up the above inequalities,
$$
|\langle \b(z+a)+\b(z-a), b\rangle-\langle \b(z+b)+\b(z-b), b\rangle|\leq  \|\b\|_{R_0}\,(|a|^2+|b|^2)
$$
and as a consequence
$$
|\langle \b(z+a)+\b(z-a)-2\b(z), b\rangle-\langle \b(z+b)+\b(z-b)-2\b(z), b\rangle|\leq \|\b\|_{R_0}\,(|a|^2+|b|^2)
$$
whence
$$\aligned
|\langle \b(z+a)+\b(z-a)-2\b(z), b\rangle|
&\leq |\langle \b(z+b)+\b(z-b)-2\b(z), b\rangle|+ \|\b\|_{R_0}\,(|a|^2+|b|^2)\\
&=|\langle \b(z-b)-\b(z), b\rangle -\langle \b(z+b)-\b(z), -b\rangle|+  \|\b\|_{R_0}\,(|a|^2+|b|^2)\\
&\leq  \|\b\|_{R_0}\,(|a|^2+2|b|^2)\\
\endaligned$$
Let us now take a vector $v\in\R^n$, and decompose it as $v=v_1+v_2$ with $v_1=\langle v, a\rangle \frac{a}{|a|^2}$. Then $\langle v_2, a\rangle=0$ so that taking $b=\frac{v_2}{|v_2|}\,|a|$ we certainly have $\langle a,b\rangle=0$ and $|a|=|b|$, and so we can apply what we proved before. Namely,
$$
\aligned
|\langle \b(z+a)&+\b(z-a)-2\b(z), v\rangle|\\
&\leq |\langle \b(z+a)+\b(z-a)-2\b(z), v_1\rangle|+ |\langle \b(z+a)+\b(z-a)-2\b(z), v_2\rangle|\\
&= |\langle \b(z+a)+\b(z-a)-2\b(z), a\rangle|\,\frac{|\langle v,a\rangle|}{|a|^2}+|\langle \b(z+a)+\b(z-a)-2\b(z), b\rangle|\,\frac{|v_2|}{|a|}\\
&\leq \|\b\|_{R_0}\,|a|^2\,\frac{|\langle v,a\rangle|}{|a|^2}+\|\b\|_{R_0}\,(|a|^2+2|b|^2)\,\frac{|v_2|}{|a|}\\
&\leq \|\b\|_{R_0}\, |\langle v,a\rangle| +\|\b\|_{R_0}\,3|a|\,|v_2|\leq 4\,\|\b\|_{R_0}\,|a|\,|v| \\
\endaligned
$$
and the claim follows.
\end{proof}

\begin{rem}
In the above proof, condition $R_0$ has only been used for precise pairs $h$ and $k$ for which either $h=k$ or $\langle h, k\rangle=0$ with $|h|=|k|$. It will be clear that the class of vector fields one obtains with this restriction is exactly the same $R_0$. This can be seen as a consequence of Theorems \ref{R0givesboundedcurl} and \ref{AboundedimpliesR0} below.
\end{rem}

\noindent
Among the consequences, we deduce that each element of $R_0$ has growth at most $|x|\,\log|x|$, as $|x|\to\infty$, and also that each element of $R_0$ has $t\,\log\frac1t$ local modulus of continuity. 
Arguing as in Reimann's Proposition 5 for $n=1$, functions in the $R_0$ class can be shown to satisfy the following extended version of condition $R_0$,
\begin{equation}\label{log}
\frac{|(\b(x+h)-\b(x))k-(\b(x+k)-\b(x))h|}{|hk|}\leq \|\b\|_{R_0}\,\left(\frac32+\frac1{2\log 2}\,\left|\log\frac{|h|}{|k|}\right|\right)
\end{equation}
provided that $h\cdot k>0$ (replace $3/2$ by $5/2$ in case you want to allow $h\cdot k<0$). The extension of this fact to functions in the higher dimensional $R_0$ class works as follows.

\begin{prop}
There exists $C=C(n)\geq 1$ such that if  $\b\in R_0$ then
$$\frac{|\langle \b(x+h)-\b(x),k\rangle -\langle \b(x+k)-\b(x),h\rangle |}{|h|\,|k|}\leq C\,\|\b\|_{R_0}\,\left(1+ \left|\log\frac{|h|}{|k|}\right|\right)
$$
whenever $h, k\in\R^n$ are non-zero.
\end{prop}
\begin{proof}
Let us fix two non-zero vectors $a,b\in\R^n$, choose $y=\frac{a}{|a|}$, and observe that
$$
\left|\frac{\langle \b(x+|a|y)-\b(x), |a|\frac{b}{|b|}\rangle}{|a|^2}-\frac{\langle \b(x+|a|\frac{b}{|b|})-\b(x),|a|y\rangle}{|a|^2}\right|\leq \|\b\|_{R_0}
$$
while
\begin{equation}\label{diff}
\left|\frac{\langle \b(x+b)-\b(x), a\rangle }{|a||b|}-\frac{\langle \b(x+|a|\frac{b}{|b|})-\b(x),|a|y\rangle}{|a|^2}\right|
=\left|\langle\frac{\b(x+b)-\b(x)}{|b|}-\frac{\b(x+|a|\frac{b}{|b|})-\b(x)}{|a|}, \frac{a}{|a|}\rangle\right|
\end{equation}
In order to control this quantity, we use the auxiliary function $g:\R\to\R$ defined as $g(t)=\langle \frac{\b(x+tb)}{|b|}, \frac{a}{|a|}\rangle$. Since $\b\in R_0$ implies $\b\in Z$, we deduce for each fixed $t$ that
$$\aligned
\left|\frac{g(t+s)+g(t-s)-2g(t)}{s}\right|
&=\left| \frac{\b(x+(t+s)b)+\b(x+(t-s)b)-2\b(x+tb),a\rangle}{s|a||b|}\right|\\
&\leq \|\b\|_Z\endaligned$$
so that $g\in Z$ with $\|g\|_Z\leq \|\b\|_Z\leq 4\|\b\|_{R_0}$. As a consequence, and arguing as in Reimann's proof of Proposition 5 (part C) we get for all $r\in\R$ that if $t,s>0$ then
$$
\left|\frac{g(r+t)-g(r)}{t}-\frac{g(r+s)-g(r)}{s}\right|\leq \|g\|_Z\,\left(\frac32+\frac1{2\log 2}\,\left|\log\frac{t}{s}\right|\right)
$$ 
In particular, if $r=0$ and $s<0<t$,
$$\aligned
\left|\frac{g(t)-g(0)}{t}-\frac{g(s)-g(0)}{s}\right|
&\leq
\left|\frac{g(t)-g(0)}{t}-\frac{g(-s)-g(0)}{-s}\right|+\left|\frac{g(-s)-g(0)}{-s}-\frac{g(s)-g(0)}{s}\right|\\
&\leq \|g\|_Z\,\left(\frac32+\frac1{2\log 2}\,\left|\log\frac{t}{-s}\right|\right)+\|g\|_Z\\
&= \|g\|_Z\,\left(\frac52+\frac1{2\log 2}\,\left|\log\frac{t}{-s}\right|\right).\endaligned$$ 
We now go back to \eqref{diff}, and apply the above estimate with $s=1$, $t=\frac{|a|}{|b|}$. We obtain
$$\aligned
\left|\langle \frac{\b(x+b)-\b(x)}{|b|}, \frac{a}{|a|}\rangle-\langle \frac{\b(x+\frac{|a|}{|b|}b)-\b(x)}{|a|}, \frac{a}{|a|}\rangle\right|
&=\left|\frac{g(t)-g(0)}{t}-\frac{g(s)-g(0)}{s}\right|\\
&\leq 4\|\b\|_{R_0}\,\left(\frac52+\frac1{2\log 2}\,\left|\log\frac{|a|}{|b|}\right|\right)
\endaligned
$$
and the claim follows.
\end{proof}


\noindent
As it was done in the previous sections for the classes $\bar{Q}$ and $R$, we are interested in a differential characterization of the class $R_0$. It is clear from Lemma \ref{pointwisecurl} that if $\b$ is a smooth element of $R_0$ then
$$\|D\b-D^t\b\|_{L^\infty}\leq \|\b\|_{R_0}$$
However, this time the situation for a non necessarily smooth $\b\in R_0$ is more delicate than in the previous sections, because differentiability points may not even exist. Indeed, with $D\b-D^t\b$ there is not enough information to control all of $D\b$. Observe also that if $n=2$ then $R_0$ is strictly larger than $R$. 

\begin{theo}\label{R0givesboundedcurl}
Let $\b\in R_0$. Then the distribution $D\b-D^t\b$ is an element of $L^\infty(\R^n)$, and
$$\|D\b-D^t\b\|_{L^\infty} \leq C(n)\,\|\b\|_{R_0}$$
for some constant $C(n)$ that depends only on $n$.
\end{theo}

\noindent
The proof of this result is very similar to that of Theorem \ref{Qimpliesboundedpartial}. However, some special attention is needed to stop the argument at an earlier point.  
 

\begin{proof}
We will first assume that $\b$ has compact support. Let us call $\u=P_y\ast \b$. We will write $\u=(u^1,\dots, u^n)$ and similarly $\b=(b^1,\dots,b^n)$. One immediately sees that $\partial_{x_i}b^j$ is a well defined distribution, because $\b$ has compact support. Moreover, this distribution can be easily extended to act against testing functions with polynomial decay, as for instance Poisson extensions of smooth compactly supported functions. So the action $\langle\partial_{x_i} b^j,  P_y\ast\varphi \rangle$ is well defined whenever $\varphi\in C^\infty_c$. One has
$$\aligned
\langle\partial_{x_i}u^j-\partial_{x_j}u^i,\varphi\rangle
&=-\langle u^j, \partial_{x_i}\varphi\rangle +\langle u^i, \partial_{x_j}\varphi \rangle\\
&=-\langle P_y\ast b^j, \partial_{x_i}\varphi\rangle +\langle P_y\ast b^i, \partial_{x_j}\varphi \rangle\\
&=-\langle  b^j, P_y\ast\partial_{x_i}\varphi\rangle +\langle  b^i, P_y\ast\partial_{x_j}\varphi \rangle\\
&=-\langle b^j, \partial_{x_i}(P_y\ast\varphi)\rangle+\langle b^i,\partial_{x_j}(P_y\ast\varphi)\rangle\\
&=\langle \partial_{x_i} b^j,  P_y\ast\varphi \rangle-\langle \partial_{x_j} b^i, P_y\ast\varphi \rangle\\
&=\langle \partial_{x_i}b^j- \partial_{x_j} b^i, P_y\ast\varphi \rangle
\endaligned$$
In particular, we have the following equality of distributions,
\begin{equation}\label{eqdistr}
\partial_{x_i}u^j-\partial_{x_j}u^i = \partial_{x_i}(P_y\ast b^j)-\partial_{x_j}(P_y\ast b^i)=P_y\ast (\partial_{x_i}b^j-\partial_{x_j}b^i)
\end{equation}
which is equivalent to say that $D\u-D^t\u=P_y\ast(D\b-D^t\b)$. The convolution operator $P_y\ast$ commutes with translations. Therefore it is not hard to see that
$$\b\in R_0\hspace{1cm}\Rightarrow\hspace{1cm}\u(\cdot, y)\in R_0,\,\,\text{and}\,\,\|\u(\cdot, y)\|_{R_0}\leq \|P_y\|_{L^1(\R^n)}\,\|\b\|_{R_0}=\|\b\|_{R_0}$$
where we used that $\|P_y\|_{L^1(\R^n)}=1$. However, $\u$ is smooth. Thus, every point $x$ is a differentiability point of $\u(\cdot, y)$, and therefore by Lemma \ref{pointwisecurl}
\begin{equation}\label{bd}
|D\u(x,y)-D^t\u(x,y)|\leq \|\u(\cdot, y)\|_{R_0}\leq \|\b\|_{R_0}.
\end{equation}
In particular, this shows that each slice of $\partial_{x_i}u^j-\partial_{x_j}u^i$ belongs to $L^\infty(\R^n)$, and this happens uniformly in $y>0$. We now show that one also has $\partial_iu^j-\partial_ju^i\in L^p(\R^n)$ for some $p\in (1,\infty)$. Indeed, from $\u=P_y\ast \b$ we see that $\partial_{x_i} \u=(\partial_{x_i}P_y)\ast \b$ and therefore
$$\aligned
|D \u(x,y)|
&\leq C\,\int_{\supp \b} |DP_y(x-z)|\,|\b(z)|\,dz \\
&= C\,\int_{\supp \b} \frac{|(n+1)c_n\,y(x-z)|}{(y^2+|x-z|^2)^\frac{n+3}{2}}\,|\b(z)|\,dz \\
&=C\,\int_{\supp \b} \frac{|(n+1)c_n\,y(x-z)|}{y^2+|x-z|^2}\,\frac{|\b(z)|}{(y^2+|x-z|^2)^\frac{n+1}{2}}\,dz\leq c_n\,\int \frac{|\b(z)|}{|x-z|^{n+1}}\,dz \\
\endaligned$$ 
As a consequence, if $\supp\b\subset B(0,R)$ and $|x|>2R$ then
$$
|\partial_{x_i}u^j(x,y)-\partial_{x_j}u^i(x,y)|\leq \frac{c_n\,\|\b\|_{L^\infty}}{|x|^{n+1}}$$
From this, if $p>\frac{n}{n+1}$ then $\|\partial_{x_i}u^j(\cdot,y)-\partial_{x_j}u^i(\cdot,y)\|_{L^p(\R^n\setminus B(0,2R))}$ is bounded uniformly in $y$. Combining this fact with \eqref{bd}, one gets that 
$$\sup_{y>0}\|\partial_{x_i}u^j(\cdot,y)-\partial_{x_j}u^i(\cdot,y)\|_{L^p(\R^n)}\leq C(n,R).$$
As a consequence, $\partial_{x_i}u^j -\partial_{x_j}u^i$ belongs to the Hardy space of harmonic functions $h^p(\R^{n+1}_+)$. As such, we can infer that there is $v_{i,j}\in L^p(\R^n)$ such that $\partial_{x_i}u^j -\partial_{x_j}u^i =P_y\ast v_{i,j}$ and moreover 
$$\lim_{y\to 0}\| (\partial_{x_i}u^j -\partial_{x_j}u^i )-(v_{i,j})\|_{L^p(\R^n)}=0.$$
In particular, there is a subsequence of heights $y_n\to 0$ for which the converge is pointwise,
\begin{equation}\label{pointwlimit}
\lim_{n\to \infty} \partial_{x_i}u^j -\partial_{x_j}u^i =v_{i,j}\hspace{1cm}a.e.
\end{equation}
which combined with \eqref{bd} gives us that $v_{i,j}\in L^\infty(\R^n)$. Finally, since \eqref{eqdistr} holds for all testing functions $\varphi\in L^{p'}(\R^n)$, we see that
$$\lim_{y\to 0}\|P_y\ast (\partial_{x_i}b^j -\partial_{x_j}b^i )-v_{i,j}\|_{L^p(\R^n)}=0$$
This forces $\|P_y\ast(\partial_{x_i}b^j -\partial_{x_j}b^i )\|_{L^p}$ to remain bounded as $y\to 0$, which in turn forces the distribution $\partial_{x_i}b^j -\partial_{x_j}b^i$ to belong to $L^p(\R^n)$, and therefore by Fatou's Theorem $v_{i,j}=\partial_{x_i}b^j -\partial_{x_j}b^i $ almost everywhere. Moreover, since $v_{i,j}\in L^\infty(\R^n)$ we also have $\partial_{x_i}b^j -\partial_{x_j}b^i \in L^\infty(\R^n)$, and 
$$\|\partial_{x_i}b^j -\partial_{x_j}b^i \|_{L^\infty}=\|v_{i,j}\|_{L^\infty}\leq \sup_{y>0}\|\partial_{x_i}u^j -\partial_{x_j}u^i \|_{L^\infty}\leq \|P_y\|_1\,\|\b\|_{R_0}$$ so the claim follows. \\
\\
In order to remove the assumption on $\supp\b$, we proceed as in Theorem \ref{Qimpliesboundedpartial}. So we start by recalling that $\b\in R_0$ implies $\b\in Z$, whence
$$L=\limsup_{|x|\to\infty}\frac{|\b(x)|}{|x|\,\log|x|}<\infty.$$
Setting $\Delta_h\varphi(x)=\varphi(x+h)-\varphi(x)$ and $\tau_h\varphi(x)=\varphi(x+h)$ and taking $g=g_t$ as in \eqref{gt} one has 
$$\aligned
\langle \Delta_h(g\b), k \rangle&-\langle \Delta_k(g\b),h\rangle\\
&=  \tau_hg\,\langle \Delta_h\b, k\rangle -\tau_kg\,\langle \Delta_k\b,h\rangle
+ \langle \b,k\rangle\,\Delta_hg-\langle \b, h\rangle\,\Delta_kg\\
&=  \tau_hg\,(\langle \Delta_h\b, k\rangle -\langle \Delta_k\b,h\rangle)+(\tau_hg-\tau_kg)\,\langle \Delta_k\b,h\rangle
+ \langle \b,k\rangle\,\Delta_hg-\langle \b, h\rangle\,\Delta_kg
\endaligned$$
If $|x|$ is large, from the mean value theorem we see that
$$|\Delta_kg(x)|\leq \frac{C|k|}{t |x|\,\log|x|}$$
$$|\Delta_hg(x)|\leq \frac{C|h|}{t |x|\,\log|x|}$$
$$|\tau_hg(x)-\tau_kg(x)|\leq \frac{C|h-k|}{t|x|\,\log|x|}$$
This, together with the growth of $\b$ at infinity, gives
$$
\left|\frac{\Delta_h(g\b),k\rangle-\langle\Delta_k(g\b),h\rangle}{|h|\,|k|}\right|\leq \|\b\|_{R_0}+\frac{C}t
$$
and so $g\b \in R_0$ and has compact support. From the first part of the proof, we deduce that $D(g\b)-D^t(g\b)\in L^\infty$, with norm les than $\|g\b\|_{R_0}$. But from
$$D(g\b )-D^t(g\b)=\b\otimes \nabla g-\nabla g\otimes \b+ g\,(D\b-D^t\b)$$
one gets at points $|x|\leq t$ that
$$D(g\b )-D^t(g\b)=D\b-D^t\b$$
whence for $|x|\leq t$ one has
$$\aligned
|D\b(x)-D^t\b(x)| 
&\leq \|D(g\b)-D^t(g\b)\|_{L^\infty}  \\
&\leq \|g\b\|_{R_0}  \\
&\leq \|\b\|_{R_0}+\frac{C}t\leq C(n) \|D\b-D^t\b\|_{L^\infty} + \frac{C}t.
\endaligned$$
The proof finishes by letting $t\to\infty$. 
\end{proof}

\noindent
Theorem \ref{R0givesboundedcurl} provides a sufficient condition for $L^\infty$ bounds for the distributional curl $D\b-D^t\b$. It says nothing about the differentiability of $\b$, nor the total pointwise differential $D\b$. For instance, if $u\in W^{1,1}(\R^n)$ and $\b=\nabla u$, then $D\b-D^t\b=0$ in the sense of distributions, but $\b$ may not be differentiable almost everywhere. That is, zero curl does not imply pointwise differentiability a.e..  Hence, at this point it is not clear why should any $\b\in R_0$ be differentiable almost everywhere. This absence of regularity makes it harder to state Theorem \ref{R0givesboundedcurl} in the same terms we stated Theorems \ref{Qimpliesboundedpartial} and \ref{Rimpliesboundeddb} above.  We solve this obstruction in the following result, which is a slight modification of  Theorem \ref{R0givesboundedcurl}. It refers to a slightly smaller subclass of $R_0$, given in terms of the divergence $\div\b$.
 
\begin{theo}\label{boundedcurl}
Let $\b:\R^n\to\R^n$ belong to the class $R_0$. Assume that $\div\b\in L^p_{loc}(\R^n)$.
\begin{itemize}
\item If $1<p<\infty$, then $\b$ has $L^p_{loc}(\R^n)$ distributional derivatives, and $\curl\b\in L^\infty(\R^n)$. 
\item If $p>n$, then one further has that $\b$ is differentiable almost everywhere.
\end{itemize}
\end{theo}
\begin{proof}
Let us first assume that $\b$ has compact support. If $\b\in R_0$ then we know from Theorem \ref{R0givesboundedcurl} that $D\b-D^t\b\in L^\infty$. Then, from the compact support we deduce that $D\b-D^t\b\in L^p$, and from $\div\b\in L^p$ and Lemma \ref{curldiv} we get that $D\b\in L^p$. The rest is standard real analysis. If $\b$ has not compact support, then using the functions $g_t$ from \eqref{gt} we see that $g_t\b$ is an element of $R_0$ with $L^p$ divergence, so using again Lemma \ref{curldiv} we get that $g_t\b$ has $L^p(\R^n)$ derivatives, which in turn ensures $D\b\in L^p_{loc}$. The differentiability a.e. is a well known result of classical real analysis, see for instance \cite{Ev}.
\end{proof}

\noindent
We also obtain the following counterpart to Theorem \ref{boundedcurl} in $\R^n$ for the case $p=\infty$. It states that vector fields in $R_0$ with bounded divergence must necessarily have $BMO$ derivatives and bounded curl. Let us recall that 
$$A\b = \frac{D\b-D^t\b}{2}+\frac{\div\b}{n}\,\Id.$$
It can also be seen as a counterpart to \cite[Proposition 15]{Rei}, as well as to Theorems \ref{Qimpliesboundedpartial} and \ref{Rimpliesboundeddb} above.

\begin{coro}
Let $\b:\R^n\to\R^n$ belong to the class $R_0$. Assume that $\div\b\in L^\infty(\R^n)$. Then $\b$ is differentiable almost everywhere, has $BMO(\R^n)$ distributional derivatives, and $A\b\in L^\infty(\R^n)$. 
\end{coro}
\begin{proof}
We first assume that that $\b$ has compact support. Having also that $\b\in R_0$, we proved in Theorem \ref{R0givesboundedcurl} that also $\curl\b\in L^\infty$. It then follows from Lemma \ref{curldiv} that $D\b\in BMO$, and so the differentiability a.e. is automatic. The boundedness of $A\b$ is immediate. The proof for non compactly supported $\b$ goes similarly, since $g_t\b$ is compactly supported and also $g_t\b\in R_0$.
\end{proof}

\noindent
In the converse direction, we have the following result, which establishes a much netter counterpart to \cite[Proposition 12]{Rei} or Theorems \ref{partialfboundedimpliesQ} or \ref{partialfboundedimpliesR}. 

\begin{theo}\label{AboundedimpliesR0}
Let $\b \in W^{1,1}_{loc}(\R^n; \R^n)$ be a vector field. Assume that $\b $ is continuous, and that $|\b (x)|\leq O( |x|\log|x|)$ as $|x|\to\infty$. If there exists a constant $C>0$ such that
$$ \| A\b  \|_{L^\infty} \leq C$$
then $\b $ belongs to the $R_0$ class, and $\|\b \|_{R_0}\leq C'$ for some constant $C'$ depending only on $C$. 
\end{theo}
\begin{proof}
Let us remind that $A\b\in L^\infty$ gives us bounds for $\div\b$ and $\curl\b$ in the $L^\infty$ norm. Again, we follow the steps in the proof of Theorem \ref{partialfboundedimpliesQ}, so we will first assume that $\b$ has compact support, and later on will remove this assumption. Given $\alpha,\beta\in \R^n$, $|\alpha|=|\beta|=1$, set $a=\alpha h, b=\beta h$ for some $h>0$. For each $\g :\R^n\to\R^n$, define
$$\Delta \g (x)=\Delta_{a,b}\g (x)=\langle \g (x+a)-\g (x),\beta\rangle -\langle \g (x+b)-\g (x),\alpha\rangle.$$
Clearly, $\Delta=\Delta_{a,b}$ is a linear operator in $\g $, and 
\begin{equation}\label{elinftybound}
|\Delta \g (x)|\leq 4\,\|\g \|_{L^\infty}
\end{equation}
Moreover, $\g $ belongs to the class $R_0$ if and only if $|\Delta \g |\leq C\,h$ for some constant $C$ that does not depend on $a$ or $b$. Using that $D\g=S\g+A\g$, we can represent $\Delta \g $  as follows, 
$$\aligned
\Delta \g (x)
&=\int_0^h \frac{d}{ds}\bigg(\langle \g (x+\alpha s), \beta\rangle-\langle \g (x+\beta s),\alpha\rangle\bigg)\,ds\\
&=\int_0^h  \langle D\g (x+\alpha s)\,\alpha, \beta\rangle-\langle D\g (x+\beta s)\,\beta,\alpha\rangle \,ds=\Delta_S\g(x)+\Delta_A\g(x)\endaligned$$
with
$$\aligned
\Delta_S\g (x)&=\int_0^h  \langle S\g (x+\alpha s)\,\alpha, \beta\rangle-\langle S\g (x+\beta s)\,\beta, \alpha\rangle\,ds\\
\Delta_A\g (x)&=\int_0^h\langle A\g (x+\alpha s)\,\alpha,\beta\rangle-\langle A\g (x+\beta s)\,\beta,\alpha\rangle \,ds
\endaligned
$$
By construction, $\u(x,y)=P_y\ast \b (x)$ is harmonic in $\R^{n+1}_+$ and continuous up to the boundary, since $\b \in C_c(\R^n)$. For each $t>0$, 
$$
\aligned 
\b (x)=\u (x,0)
&=\int_0^yt\,\partial^2_{yy}\u (x,t)\,dt-y\,\partial_y\u (x,y)+\u (x,y)\\
&\equiv\int_0^y t\,\w _t(x)\,dt-y\,\v _y(x)+\u _y(x)
\endaligned$$
where we wrote $\u _y(x)=\u (x,y)$, $\v _y(x)= \partial_y\u (x,y)$ and $\w _r(x)= \partial_{yy}^2\u (x,r)$. By the linearity of $\Delta$, which acts only on the $x$ variable, one has
$$
\Delta \b (x)=\int_0^y t\, \Delta \w _t(x)\,dt-y\,\Delta \v _y(x)+\Delta \u _y(x).
$$
We now proceed term by term. First, from Lemma \ref{curldiv} we know that $A\b \in L^\infty$ implies $D\b \in BMO$, which in turn gives $\b \in Z$. Hence, from Lemma \ref{poissonZ},
$$
\left|\int_0^y t\, \Delta \w _t(x)\,dt\right|\leq \int_0^y t\,4\|\w _t\|_{L^\infty}\,dt\leq \int_0^y t\,4\,\frac{C(n)\,\|\b \|_Z}{t}\,dt= C(n)\,y\,\|\b \|_Z.
$$
For the second term, we use that $\Delta =\Delta_S+\Delta_A$,  
$$\aligned
y\,\Delta \v _y(x)&=y\,\Delta_S \v _y(x)+y\,\Delta_A\v _y(x)\\
 \Delta \u _y(x)&= \Delta_S \u _y(x)+ \Delta_A\u _y(x)\\
\endaligned
$$
and proceed first with the $\Delta_A$ terms. For each fixed $y$, Lemma \ref{easyLinfty} gives us that
$$ \aligned
\partial_{x_i}\,\u _y=  \partial_{x_i}  \,(P_y\ast \b )=P_y\ast (\partial_{x_i}\b )
&\Longrightarrow A\u _y=P_y\ast A\b \\
&\Longrightarrow \|A\u _y\|_{L^\infty}=\|P_y\ast A\b \|_{L^\infty}\leq \|P_y\|_1\,\|A\b \|_{L^\infty}=\|A\b \|_{L^\infty}\endaligned$$
On the other hand, since $\u $ is smooth, we can argue similarly to get that
$$\aligned
\partial_{x_i} \v _y = \partial^2_{y, x_i} \u  =  \partial_y\,\left(P_y\ast \partial_{x_i} \b \right)
&\Longrightarrow A\v _y = \partial_y(P_y\ast A\b )\\
&\Longrightarrow \|A\v _y\|_{L^\infty}  \leq C(n)\,\frac{\|A\b \|_{L^\infty}}{y}.
\endaligned
$$
Thus
$$
\aligned
|\Delta_A \u _y(x)|&\leq 2h\,\|A\u _y\|_{L^\infty} \leq 2h\,\|A\b \|_{L^\infty}\\
|y\,\Delta_A\v _y(x)|&\leq 2h y \|A\v _y\|_{L^\infty}\leq C(n)\,h\,\|A\b \|_{L^\infty}
\endaligned
$$
for some dimensional constant $C(n)$. Now is time to proceed with the $\Delta_S$ terms. For any function $\g$, set
$$(S\g)_{\alpha,\beta}(x)=\langle S\g(x)\cdot\alpha, \beta\rangle.$$
Using that $S\g$ is a symmetric matrix, and calling $\gamma=\frac{\alpha-\beta}{|\alpha-\beta|}$,
\begin{equation}\label{Scommut}
\aligned
\langle S\g(x+\alpha s)\,\alpha, \beta\rangle-\langle S\g(x+\beta s)\,\beta, \alpha\rangle
&=\langle\alpha, (S\g(x+\alpha s)-S\g(x+\beta s))\,\beta\rangle\\
&=\langle\alpha,\left(\int_0^{s|\alpha-\beta|}\frac{d}{d\sigma} (S\g(x+\beta s+\sigma\gamma))\,d\sigma\right)\beta\rangle\\
&= \int_0^{s|\alpha-\beta|}\frac{d}{d\sigma} \bigg(\langle\alpha,S\g(x+\beta s+\sigma\gamma)\,\beta\rangle\bigg) d\sigma\\
\endaligned
\end{equation}
Therefore
$$
\aligned
|\Delta_S \g(x)| 
&\leq \int_0^h \int_0^{s|\alpha-\beta|}\left|\frac{d}{d\sigma} (S\g)_{\alpha,\beta}(x+\beta s+\sigma\gamma)   \right|\,d\sigma\, ds\\
&\leq \int_0^h \int_0^{s|\alpha-\beta|} | D((S\g)_{\alpha,\beta})(x+\beta s+\sigma\gamma) |\,d\sigma\, ds\\
&\leq \| D((S\g)_{\alpha,\beta})\|_{L^\infty}\,\int_0^h \int_0^{s|\alpha-\beta|} d\sigma\, ds= \| D((S\g)_{\alpha,\beta})\|_{L^\infty}\,\frac{h^2\,|\alpha-\beta| }2
\endaligned
$$
After applying this to $\g=\u _y$ and to $\g=\v _y$, one obtains 
\begin{equation}\label{finalstep}
|\Delta \b (x)|\leq C(n)\,y\,\|\b \|_Z+C(n)\,h\,\|A\b \|_{L^\infty}+\frac{h^2\,|\alpha-\beta| }2\left(\| D((S\u _y)_{\alpha,\beta})\|_{L^\infty}+y\| D((S\v _y)_{\alpha,\beta}) \|_{L^\infty}\right)
\end{equation}
Next, we see that
$$
\aligned
\u _y=P_y\ast \b  \hspace{1cm}
&\Rightarrow\hspace{1cm}D\u _y= P_y\ast D\b \\
&\Rightarrow\hspace{1cm} S\u _y=P_y\ast S\b \\
&\Rightarrow\hspace{1cm}(S\u _y)_{\alpha,\beta}=P_y\ast (S\b )_{\alpha, \beta}.\endaligned$$
Now, since $D\b \in BMO$ we have in particular that $(S\b )_{\alpha,\beta}\in BMO$, in particular $(S\u _y)_{\alpha,\beta}$ is harmonic Bloch. Lemma \ref{poissonBMO} with $\g= (S\u _y)_{\alpha,\beta}$ gives us that
\begin{equation}\label{sb1}
\|D((S\u _y)_{\alpha,\beta})\|_{L^\infty}=\|D(P_y\ast (S\b )_{\alpha,\beta})\|_{L^\infty}\leq \frac{C(n)\,\|(S\b )_{\alpha,\beta}\|_\ast}{y}\leq \frac{C(n)\,\|S\b \|_\ast}{y}.
\end{equation}
For $\g=(S\v _y)_{\alpha,\beta}$, we proceed similarly and note that 
$$
\aligned
\v _y=\partial_y\u _y=\partial_yP_y\ast \b \hspace{1cm}
&\Rightarrow\hspace{1cm}D\v _y= \partial_yP_y\ast D\b \\
&\Rightarrow\hspace{1cm} S\v _y=\partial_yP_y\ast S\b \\
&\Rightarrow\hspace{1cm}(S\v _y)_{\alpha,\beta}=\partial_yP_y\ast (S\b )_{\alpha, \beta}.\endaligned$$
Therefore one can combine Lemma \ref{poissonBMO} and Lemma \ref{higherorder} and obtain
\begin{equation}\label{sb2}
\|D((S\v _y)_{\alpha,\beta})\|_{L^\infty}=\|D(\partial_yP_y\ast(S\b )_{\alpha,\beta})\|_{L^\infty}\leq\frac{C(n)\,\|(S\b )_{\alpha,\beta}\|_\ast}{y^2}\leq \frac{C(n)\|S\b \|_\ast}{y^2}
\end{equation}
It is worth mentioning here that both in \eqref{sb1} and \eqref{sb2} one could replace the constant  $\|S\b \|_\ast$ by $\|\b\|_Z$ (note that $\|\b\|_Z\leq C\,\|Sb\|_\ast$). To do this, one only needs to use Lemma \ref{poissonZ} instead of Lemma \ref{poissonBMO}. We now plug the above bounds for $\|D((S\u _y)_{\alpha,\beta})\|_{L^\infty}$ and $\|D((S\v _y)_{\alpha,\beta})\|_{L^\infty}$ into \eqref{finalstep}, and then take $h=y$. This finishes the proof in the case $\b \in C_c(\R^n)$.\\
In order to remove the assumption on the compact support, we use once more the auxilliary function $g=g_t$ introduced at \eqref{gt}. We have
$$
D(g\b  )-D^t(g\b )=\b \otimes \nabla g-\nabla g\otimes \b + g\,(D\b -D^t\b )$$
so
$$\aligned
\|D(g\b )-D^t(g\b )\|_{L^\infty} 
&\leq \|D\b -D^t\b \|_{L^\infty} + \sup_{t\leq|x|\leq t^{e^t}}|\b (x)||\nabla g(x)|\\
&\leq \|D\b -D^t\b \|_{L^\infty} + \sup_{t\leq|x|\leq t^{e^t}}C\,|x|\log|x|\frac{1}{t|x|\,\log|x|}\\
&\leq \|D\b -D^t\b \|_{L^\infty} + \frac{C}t\\
\endaligned$$
Now the claim follows since for every $x\in\R^n$ we can pick $t>0$ large enough and such that $|x|< t^{e^t}$ so that $\b =g\b $ in a neighbourhood of $x$, and therefore $D\b -D^t\b =D(g\b )-D^t(g\b )$.
\end{proof}

 \noindent
 Albert Clop\\
 Department of Mathematics and Computer Science\\
 Universitat de Barcelona\\
 08007-Barcelona\\
 CATALONIA\\
 albert.clop@ub.edu\\
 \\
 \\
 Banhirup Sengupta\\
 Departament de Matem\`atiques\\
 Unirersitat Aut\`onoma de Barcelona\\
 08193-Bellaterra\\
 CATALONIA\\
 sengupta@mat.uab.cat\\
 
 \end{document}